\documentclass{article}

\usepackage{amsfonts,amsmath,amssymb,amsthm}
\usepackage{graphicx,xcolor}
\usepackage[margin=1in]{geometry}
\usepackage{tikz}
\usepackage{blkarray}
\usepackage{enumitem}

\newtheorem*{theorem*}{Theorem}
\newtheorem{theorem}{Theorem}
\newtheorem{lemma}[theorem]{Lemma}
\newtheorem{proposition}[theorem]{Proposition}
\newtheorem{corollary}[theorem]{Corollary}
\newtheorem{conjecture}[theorem]{Conjecture}
\theoremstyle{definition}
\newtheorem{definition}[theorem]{Definition}
\theoremstyle{remark}
\newtheorem{remark}[theorem]{Remark}

\newtheorem{example}[theorem]{Example}

\numberwithin{theorem}{section}

\newcommand\cA{{\mathcal A}}

\newcommand\cC{{\mathcal C}}

\newcommand\cH{{\mathcal H}}

\newcommand\cM{{\mathcal M}}

\newcommand\cR{{\mathcal R}}

\newcommand\NN{{\mathbb N}}

\newcommand\RR{{\mathbb R}}

\DeclareMathOperator{\cl}{cl}

\DeclareMathOperator{\rank}{rank}
\DeclareMathOperator{\supp}{supp}

\newcommand{\added}[1]{{#1}}
\newcommand{\arxiv}[1]{{#1}} %Stuff that was arxiv only

\title{$k$-fold circuits and coning in rigidity matroids}
\author{John Hewetson\thanks{School of Mathematical Sciences, Lancaster University} \and Bill Jackson\thanks{School of Mathematical Sciences, Queen Mary University of London, E-mail: \texttt{b.jackson@qmul.ac.uk}} \and Anthony Nixon\thanks{School of Mathematical Sciences, Lancaster University, E-mail: \texttt{a.nixon@lancaster.ac.uk}} \and Ben Smith\thanks{School of Mathematical Sciences, Lancaster University, E-mail: \texttt{b.smith9@lancaster.ac.uk}}}
\date{\today}

\begin{document}

\maketitle

\begin{abstract}
In 1980 Lov\'{a}sz introduced the concept of a double circuit in a matroid. The 2nd, 3rd and 4th authors recently generalised this notion to $k$-fold circuits (for any natural number $k$) and proved foundational results about these $k$-fold circuits. 
In this article we use $k$-fold circuits to derive new results on the generic $d$-dimensional rigidity matroid $\cR_d$. These results include analysing 2-sums, showing sufficient conditions for the $k$-fold circuit property to hold for $k$-fold $\cR_d$-circuits, and giving an extension of Whiteley's coning lemma. The last of these allows us to reduce the problem of determining if a graph $G$ with a vertex $v$ of sufficiently high degree  
is independent in $\cR_d$ to that of verifying matroidal properties of  $G-v$ in $\cR_{d-1}$.
\end{abstract}

\section{Introduction}

A bar-joint \emph{framework} $(G,p)$ is the combination of a finite, simple graph $G=(V,E)$ and a realisation $p:V\rightarrow \mathbb{R}^d$ which realises the vertices of $G$ in Euclidean space. The framework is \emph{rigid in $\mathbb{R}^d$} if the only edge-length preserving continuous motions of the vertices arise from isometries of $\mathbb{R}^d$.
The framework is \emph{minimally rigid in $\mathbb{R}^d$} if it is rigid in $\RR^d$, but $(G - e, p)$ is not rigid in $\RR^d$ for any edge $e \in E$.
A framework $(G,p)$ is \emph{generic} if the set of coordinates of $p$ form an algebraically independent set over $\mathbb{Q}$.
%We will see in Section~\ref{sec:rigidity+prelims} that the edge sets of graphs $G=(V,E)$ which have minimally rigid generic realisations in $\mathbb{R}^d$ are the bases of a matroid $\cR_d$ on the edge set of the complete graph \added{$K_V$} on $V$.
% We will see in Section~\ref{sec:rigidity+prelims} that the edge sets of graphs $G$ which have minimally rigid generic realisations in $\mathbb{R}^d$ are the bases of a matroid \added{$\cR_d(K_V)$} on the edge set of the complete graph \added{$K_V$ on vertex set $V$ where $V(G) \subseteq V$.}
% \added{Setting $V$ to be countably infinite allows us to consider all minimally rigid graphs simultaneously: we denote the matroid $\cR_d := \cR_d(K_V)$ the \emph{full $d$-dimensional rigidity matroid}.}

\added{Given a graph $G = (V,E)$, the edge sets of subgraphs which have minimally rigid generic realisations in $\mathbb{R}^d$ are the bases of a matroid $\cR_d(G)$ called the \emph{$d$-dimensional rigidity matroid} of $G$.
We often focus on $\cR_d := \cR_d(K_n)$ whose underlying ground set is the edge set of a sufficiently large complete graph.}
It is of central theoretical and applied importance to provide a combinatorial understanding of $\cR_d$.
It is folklore that $\cR_1$ is precisely the cycle matroid of the underlying graph.
There is also a combinatorial description of independence in $\cR_2$ due to Pollaczek-Geiringer \cite{pol}. This result is often referred to as Laman's theorem due to an independent rediscovery \cite{laman}. When $d\geq 3$, no such characterisation is known. 

We will take a matroidal approach using ideas from our recent paper \cite{JNS}.
Let $\cM=(E,r)$ be a matroid with finite ground set $E$ and rank function $r$. A {\em circuit} of $\cM$ is a set $C\subseteq E$ such that $r(C)=|C|-1=r(C-e)$ for all $e\in E$. 
In 1980 Lov\'{a}sz \cite{Lov}, motivated by the matroid matching problem, introduced the concept of a double circuit to understand sets with multiple dependencies. A {\em double circuit} of $\cM$ is a set $D\subseteq E$ such that $r(D)=|D|-2=r(D-e)$ for all $e\in E$. 
For example, if $\cM$ is a cycle matroid (or graphic matroid) of a graph $G$, then the circuits of $\cM$ are the edge sets of cycles in $G$, and the double circuits are the edge sets of subgraphs of $G$ which are pairs of cycles with at most one vertex in common or \emph{theta graphs}, i.e. graphs consisting of 3 internally disjoint paths between two vertices. 
In general, a double circuit $D$ can contain arbitrarily many circuits.
A key insight of Lov\'{a}sz was the concept of the \emph{principal partition}, \added{a partition $\{A_1, \dots, A_\ell\}$ of $D$ where the circuits in $D$ are precisely $D \setminus A_i$ for $1 \leq i \leq \ell$.}
%of a double circuit which precisely describes all of its circuits. 

In \cite{JNS} we generalised double circuits to {\em $k$-fold circuits} i.e. sets $D\subseteq E$ such that $r(D)=|D|-k=r(D-e)$ for all $e\in D$, for some fixed integer $k\geq 0$.
\added{For example, if $\cM$ is a cycle matroid of a graph $G$, then an example of a $k$-fold circuit would be a subgraph consisting of $k+1$ internally disjoint paths between two vertices.
We also introduced the notion of the principal partition of a $k$-fold circuit as the partition $\{A_1, \dots, A_\ell\}$ where the $(k-1)$-fold circuits in $D$ are precisely $D \setminus A_i$ for $1 \leq i \leq \ell$.}
Our motivation to introduce and study $k$-fold circuits is to provide a better understanding of $\cR_d$.

\added{Dress and Lov\'asz introduced the \emph{double circuit property} in their study of the matroid matching problem \cite{DL}.}
This property was generalised to the $k$-fold circuit property in \cite{JNS}. 
A $k$-fold circuit $D$ with principal partition $\{A_1, \dots, A_\ell\}$ in a matroid $\cM$ is \emph{balanced} if 
\[
r\left(\bigcap_{i=1}^\ell \cl(D \setminus A_i)\right) = \ell - k \, ,
\]
and $\cM$ has the \emph{$k$-fold circuit property} if all of its $k$-fold circuits are balanced.
\added{This property was introduced as a measure of how close the lattice of flats of a matroid is to being a modular lattice.
For example, the $k$-fold circuit property is satisfied by pseudomodular matroids for all natural numbers $k \geq 2$.
In comparison, sparse paving matroids satisfy the $k$-fold circuit property for all $k \geq 3$ but do not satisfy the double circuit property, as the circuits within a double circuit can have unpredictable intersections.
To evaluate their combinatorial complexity, we study the $k$-fold circuit property in rigidity matroids.}
Makai showed that $\cR_1$ and $\cR_2$ have the double circuit property \cite{Mak}.
In addition, $\cR_1$ has the $k$-fold circuit property for all $k\geq 2$ by \cite{JNS}.
Conversely, we show that $\cR_d$ does not have the $k$-fold circuit property for any $d \geq 4$ and $k \geq 2$ (Theorem \ref{thm:K_67+dcp}), \added{demonstrating its lattice of flats is very far from modular.}
We also obtain two sufficient conditions for a particular $k$-fold $\cR_d$-circuit to be balanced \added{for any dimension $d$} (Theorem~\ref{thm:atmost2technicolour} and Theorem~\ref{thm:balanced+rigid+k+circuit}).
\added{For all $k \geq 2$, whether $\cR_3$ has the $k$-fold circuit property remains wide open.
As $\cR_3$ is conjectured to be equal to the generic cofactor matroid $C_2^1$, this presents a possible distinguishing invariant between the two matroids if the conjecture does not hold \cite{CJJT}.}

\added{Dress and Lov\'asz observed that for matroids satisfying the double circuit property, a max-min formula holds for the matroid matching problem.
This is not a necessary condition, and so matroids for which this max-min formula holds are said to have the \emph{matching matching property} (see Section \ref{sec:matroid+matching} for a formal definition).
Despite being NP-hard in general, the matroid matching problem can be solved in polynomial time for matroids satisfying the matroid matching property \cite{Lov}.
Even for matroids that do not satisfy it, one can often embed them inside a larger matroid with the matroid matching property to gain algorithmic advantages.
For example, any representable matroid can be embedded inside the `full linear matroid' whose ground set is the set of all vectors in a vector space.

While rigidity matroids $\cR_d(G)$ are always representable, it would be algorithmically preferable to embed them inside a smaller, combinatorial matroid.
For example, both graphic and transversal matroids can be embedded inside combinatorial `full' matroids, and both have more efficient algorithms for the matroid matching problem than the general representable case \cite{Lov,TLV}.
It is therefore natural to ask whether $\cR_d$ has the matroid matching property.
We show that $\cR_d$ does not have the matroid matching property for $d \geq 4$ (Theorem \ref{thm:matroid+matching}).}

\added{We also use $k$-fold circuits to analyse the coning operation from rigidity theory.}
Given a graph $G = (V,E)$,  the \emph{cone of $G$ over a new vertex $v$} is the graph $G*v$ obtained by adding $v$ and joining it to every vertex of $G$.
Whiteley showed that $G$ is (minimally) rigid in $\mathbb{R}^d$ if and only if $G*v$ is (minimally) rigid in $\mathbb{R}^{d+1}$ \cite{Wcone}.
This allows us to reduce the problem of characterising the rigidity of graphs with a `cone vertex' to a lower dimension.
To widen this family, we consider the `almost cone' $G' \subseteq G*v$ of $G$, where we join $v$ to all but one or two of the vertices of $G$.
We characterise when $G'$ is minimally rigid in $\mathbb{R}^{d+1}$ in terms of matroid properties of $G$ in $\cR_d$ (Theorem~\ref{thm:maint=2}).
This is proved using our results on
double $\cR_d$-circuits and how they behave under coning.
As corollaries, we show that if $G$ is independent in $\cR_d$ then we can add any two edges to $G$ and remain independent in $\cR_{d+1}$, and deduce a special case of the well known $X$-replacement conjecture\footnote{See Conjecture \ref{con:xrep} for a formal statement of this conjecture. Note that, while the $X$-replacement conjecture is known to be false for all $d\geq 4$, our special case holds for all $d\geq 3$.} for minimally rigid graphs in $\mathbb{R}^d$.

The paper is organised as follows. In Section \ref{sec:prelim} we review the necessary preliminaries from rigidity theory and matroid theory.
Section \ref{sec:kfold_in_rigidity_matroids} considers the $k$-fold circuit property for rigidity matroids. In Theorems \ref{thm:atmost2technicolour} and \ref{thm:balanced+rigid+k+circuit} we obtain sufficient conditions for a particular $k$-fold circuit in the generic $d$-dimensional rigidity matroid \added{to be balanced}.
We also show that in general, when $d\geq 4$, the rigidity matroid does not satisfy the double circuit property and show further that it does not have the matroid matching property. 

In Section \ref{sec:coning} we \added{present some applications to rigidity theory. Most notably we consider} the coning operation (which adds one new vertex to a graph and joins it to all existing vertices) and its effect on the generic $d$-dimensional rigidity matroid. 
We show in Theorem \ref{thm:coning-k-circuits-strong} that a graph is a $k$-fold $\cR_d$-circuit if and only if its cone is a $k$-fold $\cR_{d+1}$-circuit, and \added{give a partial characterisation of} how the principal partition changes under coning.
Among other applications to rigidity we reduce the problem of determining if a graph $G=(V,E)$ with a vertex of degree at least $|V|-3$  
is independent in $\cR_d$ to that of verifying matroidal properties of $G-v$ in $\cR_{d-1}$.
We conclude with some brief comments on possible future directions in Section \ref{sec:conclusion}.

\section{Preliminaries}
\label{sec:prelim}

We first introduce the relevant terminology and results from 
matroid theory and rigidity theory. 

\subsection{Matroids}

We refer the reader to \cite{Oxl} for basic definitions and concepts for matroids.

Let $\cM = (E,r)$ be a matroid on ground set $E$ with rank function $r$.
The \emph{closure operator} of $\cM$ is the function $\cl \colon 2^E \rightarrow 2^E$ defined by $\cl(X) = \{x \in E \colon r(X + x) = r(X)\}$.
A \emph{flat} of $\cM$ is a subset $F \subseteq E$ such that $\cl(F) = F$.
A pair of flats $X,Y$ are \emph{modular} if
\[
r(X) + r(Y) = r(X \cup Y) + r(X \cap Y) \, .
\]
A \emph{cyclic set} of $\cM$ is a subset $D \subseteq E$ which is the union of circuits of $\cM$. Equivalently $D$ is cyclic if $r(D-e)=r(D)$ for all $e\in D$. 

Given two matroids $\cM_1=(E_1,r_1)$ and $\cM_2=(E_2,r_2)$ with $E_1\cap E_2=\emptyset$ we define their {\em direct sum} to be the matroid $\cM_1\oplus \cM_2=(E_1\cup E_2,r)$ by putting $r(A)=r_1(A\cap E_1)+r_2(A\cap E_2)$ for all $A\subseteq E_1\cup E_2$.

We can define a relation on the ground set of a matroid $\cM=(E,r)$ by saying that $e,f \in E$ are related if $e=f$ or if there is a circuit $C$ in $\cM$ with $e,f \in C$. It is well-known that this is an equivalence relation and that, if 
its equivalence classes are $E_1, E_2, \dots, E_t$, then  $\cM=\cM_1\oplus \cM_2\oplus \dots \oplus \cM_t$, where $\cM_i=\cM|_{E_i}$ is the matroid restriction of $\cM$ onto $E_i$, see \cite{Oxl}. The classes $E_1, E_2, \dots, E_t$ are called the \emph{components} of $\cM$.
The matroid $\cM$ is \emph{connected} if it has only one component  and otherwise it is said to be \emph{disconnected}.
 We will say that a set $S\subseteq E$ is  {\em $\cM$-connected} if $\cM|_S$ is a connected matroid. 

A matroid $\cM$ is completely determined by its set of circuits $\cC(\cM)$. We can use this fact to 
define our final two matroid constructions.
Suppose $\cM_1,\cM_2$  are two matroids with ground sets $E_1,E_2$ respectively such that $E_1\cap E_2=\{e\}$ and $e$ is neither a loop nor coloop of $\cM_1$ or $\cM_2$.
The {\em parallel connection of $\cM_1,\cM_2$ along $e$} is the matroid
$P(\cM_1, \cM_2)$ with ground set $E_1\cup E_2$ and circuits
\begin{align} \label{eq:parallel+connection+circuits}
\cC(P(\cM_1, \cM_2))= \cC(\cM_1)\cup \cC(\cM_2) \cup \{(C_1\cup C_2)-e:e\in C_i\in \cC(\cM_i) \mbox{ for both } i=1,2\}.
\end{align}
The {\em $2$-sum of $\cM_1,\cM_2$ along $e$} is the matroid
$\cM_1\oplus_2 \cM_2$ with ground set $(E_1\cup E_2)-e$ and circuits
\begin{align} \label{eq:2sum+circuits}
\cC(\cM_1\oplus_2 \cM_2)=\{C\in \cC(\cM_1)\cup \cC(\cM_2) :e\not\in C\}\cup \{(C_1\cup C_2)-e:e\in C_i\in \cC(\cM_i) \mbox{ for both } i=1,2\}.
\end{align}
It follows directly from these definitions that the 2-sum of $\cM_1,\cM_2$ along $e$ is obtained by deleting $e$ from the parallel connection of $\cM_1,\cM_2$, i.e. $\cM_1\oplus_2 \cM_2 = P(\cM_1, \cM_2) \setminus e$.

We will need the following description of the flats of 2-sums and parallel connections.
A set $F \subseteq E$ is a flat of $P(\cM_1, \cM_2)$ if and only if $F \cap E_i$ is a flat of $\cM_i$ for both $i=1,2$, see~\cite[Proposition 7.6.6]{Brylawski}.
Using $\cM_1\oplus_2 \cM_2 = P(\cM_1, \cM_2) \setminus e$, it is straightforward to check that $F \subseteq E - e$ is a flat of $\cM_1\oplus_2 \cM_2$ if and only if $F \cap (E_i - e)$ is a flat of $\cM_i$ for $i=1,2$, or $(F \cup \{e\}) \cap E_i$ is a flat of $\cM_i$ for $i=1,2$.

Parallel connections and 2-sums have some common properties described in the following lemma.

\begin{lemma}\label{lem:2summat}
Let $\cM_1, \cM_2$ be matroids with ground sets $E_1, E_2$ such that $|E_i| \geq 2$ and $E_1 \cap E_2 = \{e\}$.
Suppose $\cM \in \{P(\cM_1, \cM_2), \cM_1\oplus_2 \cM_2\}$. Then:
\begin{enumerate}
    \item[(a)] $\cM$ is connected if and only if $\cM_1,\cM_2$ are both connected;
    \item[(b)] $\rank \cM=\rank \cM_1 + \rank \cM_2-1$.
\end{enumerate}
\end{lemma}

\begin{proof}
(a) is given in \cite[Proposition 7.1.17]{Oxl} and \cite[Proposition 7.1.22 (ii)]{Oxl}. \\
(b) follows from \cite[Proposition 7.1.15 (i)]{Oxl}  for parallel connections.
As $\cM_1 \oplus_2 \cM_2 = P(\cM_1, \cM_2) \setminus e$ and $e$ is not a coloop, the result follows for 2-sums. \\
\end{proof}

\subsection{Rigidity theory} \label{sec:rigidity+prelims}

We first give a formal definition of $\cR_d$ as the row matroid of a matrix.

When a framework $(G,p)$ is generic, we can perform a standard linearisation technique on the length constraints to study its rigidity.
More precisely, we define the \emph{rigidity matrix} $R(G,p)$ of a $d$-dimensional  framework $(G,p)$ to be the $|E|\times d|V|$ matrix whose rows are indexed by the edges and $d$-tuples of columns indexed by the vertices. The row for an edge $e=uv$ is given by:
$$\begin{pmatrix} 0 & \dots & 0 & p(u)-p(v) & 0 & \dots & 0 & p(v)-p(u) & 0 & \dots & 0\end{pmatrix}$$ 
where $p(u)-p(v)$ occurs in the $d$-tuple of columns indexed by $u$, $p(v)-p(u)$ occurs in the $d$-tuple of columns indexed by $v$ and $p(u),p(v)\in \mathbb{R}^{d}$.
Maxwell  \cite{Max} observed that
%It is easy to see that 
$\mbox{rank } R(G,p)\leq d|V|-\binom{d+1}{2}$ whenever $p$ affinely spans $\mathbb{R}^d$, and  a fundamental result of Asimov and Roth \cite{asi-rot} tells us that, when $p$ is generic and $G$ has at least $d+1$ vertices, $(G,p)$ is rigid if and only if rank $R(G,p)=d|V|-\binom{d+1}{2}$.

The {\em $d$-dimensional rigidity matroid} of a graph $G=(V,E)$ is the matroid $\mathcal{R}_d(G)$ on $E$ in which a set of edges $F\subseteq E$ is independent whenever the corresponding rows of $R(G,p)$ are independent, for some (or equivalently every) generic $p:V\to \mathbb{R}^d$. 
\added{Note that if $H \subseteq G$, the rigidity matroid $\cR_d(H)$ is the matroid restriction of $\cR_d(G)$ to the ground set $E(H)$.
We denote the rank function of $\cR_d(G)$ by $r_d$: as rank does not change under restriction, this does not depend on $G$.
We will generally consider the rigidity matroids of complete graphs $\cR_d(K_n)$ as this captures all rigidity properties of graphs on that vertex set. 
We denote the closure operator of $\cR_d(K_n)$ by $\cl_d$: unlike rank, closure is not preserved under matroid restriction, hence we will only ever take closures in $\cR_d(K_n)$.
To simplify notation, we will write $r_d(G)$ as shorthand for $r_d(E(G))$, and $\cl_d(G) = (V(G),\cl_d(E(G)))$ for the closure of the graph.}
% \added{Rather than specify the ground set, we will often work with the \emph{full $d$-dimensional rigidity matroid} $\mathcal{R}_d := \cR_d(K)$ where $K$ is the complete graph on a countable infinite vertex set.
% As any finite graph $G$ can be embedded into $K$, it follows that $\mathcal{R}_d(G)$ is the matroid restriction of $\mathcal{R}_d$ to the ground set $E(G)$.
% We denote the rank function of $\cR(G)_d$ by $r_d$ and its closure operator by $\cl_d$.
% Moreover, we will simplify notation by writing $r_d(G)$ as shorthand for $r_d(E(G))$.
% Similarly, we write $\cl_d(G) = (V(G),\cl_d(E(G)))$ for the closure of the graph.}

We will simplify terminology by describing $G$ using properties of its edge set in $\cR_d(G)$. 
For example we say that  $G$ is: \emph{$\mathcal{R}_d$-independent} if $r_d(G)=|E|$; \emph{$\mathcal{R}_d$-rigid} if $G$ is a complete graph on at most $d+1$ vertices or $r_d(G)=d|V|-{d+1\choose 2}$; \emph{minimally $\mathcal{R}_d$-rigid} if $G$ is $\mathcal{R}_d$-rigid and $\mathcal{R}_d$-independent; and an \emph{$\mathcal{R}_d$-circuit} if $G$ is not $\mathcal{R}_d$-independent but $G-e$ is $\mathcal{R}_d$-independent for all $e\in E$. We also say that an edge $e$ of $G$ is an \emph{$\cR_d$-bridge} (or coloop) in $G$ if $r_d(G-e) = r_d(G)-1$ holds and that $G$ is \emph{$\cR_d$-flexible} if it is not $\cR_d$-rigid.

The above mentioned observation of Maxwell on the rank of the rigidity matrix implies the following result, 
see \cite[Lemma 11.1.3]{Wlong}.

\begin{lemma}\label{lem:max}
 Let $G=(V,E)$ be a graph. \begin{enumerate}
 \item If $G$ is  $\cR_d$-independent with $|V|\geq d$, then $|E|\leq d|V|-{d+1\choose 2}$. 
 \item If $G$ is an $\cR_d$-circuit then $|E|\leq d|V|-{d+1\choose 2}+1$ with equality if and only if $G$ is $\cR_d$-rigid. 
 \end{enumerate}
\end{lemma}

The next lemma lists some scenarios where we can deduce rigidity or independence of a graph by properties of its subgraphs.

\begin{lemma}[{\cite[Lemma 11.1.9]{Wlong}}]\label{lem:intbridge}
Let $G_1$, $G_2$ be subgraphs of a graph $G$ and suppose that $G=G_1\cup G_2$.
\begin{enumerate}
\item\label{it:intbridge:rig} 
	If $|V(G_1)\cap V(G_2)|\geq d$ and $G_1,G_2$ are $\mathcal{R}_d$-rigid then $G$ is $\mathcal{R}_d$-rigid.
\item\label{it:intbridge:indep} 
	If $G_1\cap  G_2$ is $\mathcal{R}_d$-rigid and $G_1,G_2$ are $\mathcal{R}_d$-independent then $G$ is $\mathcal{R}_d$-independent.
\item\label{it:intbridge:rank} 
	If $|V(G_1)\cap V(G_2)| \leq d-1$, $u\in V(G_1)-V(G_2)$ and $v\in V(G_2)-V(G_1)$ then
	$r_d(G+uv)=r_d(G)+1$.
\end{enumerate}
\end{lemma}

We next discuss some common graph operations used for deducing rigidity.
A graph $G'$ is said to be obtained from another graph $G$ by: a \emph{($d$-dimensional) 0-extension}
if $G=G'-v$ for a vertex  $v\in V(G')$ with $d_{G'}(v)=d$; or a \emph{($d$-dimensional) 1-extension} if $G=G'-v+xy$ for a vertex $v\in V(G')$ with $d_{G'}(v)=d+1$ and $x,y\in N_{G'}(v)$ \added{with $xy \notin E(G')$}.
It is well known that both of these operations preserve $\cR_d$-independence.

\begin{lemma}\label{lem:01ext}\cite[Lemma 11.1.1, Theorem 11.1.7]{Wlong}
Let $G$ be $\cR_d$-independent and let $G'$ be obtained from $G$ by a 0-extension or a 1-extension. Then $G'$ is $\cR_d$-independent.
\end{lemma}

Recall that the cone of a graph $G=(V,E)$ over a new vertex $v$ is the graph
$G*v = (V \cup v, E \cup \{uv \mid u \in V\})$ obtained by adding $v$ and joining it to every vertex of $G$.
The next two lemmas characterise the rigidity matroid of $G*v$ in $(d+1)$-dimensions in terms of the rigidity matroid of $G$ in $d$-dimensions.

\begin{lemma}[\cite{Wcone}]\label{lem:coning+rank}
Let $G*v$ be the cone of a graph $G$ over a new vertex $v$. Then 
\[
r_{d+1}(G*v) = r_d(G) + |V(G)| \, .
\]
In particular, $G$ is $\mathcal{R}_d$-independent if and only if $G*v$ is $\mathcal{R}_{d+1}$-independent.
\end{lemma}
\begin{lemma}[\cite{GGJ}]\label{lem:coning+circuit}
Let $G*v$ be the cone of a graph $G$ over a new vertex $v$. Then $G$ is an $\mathcal{R}_d$-circuit if and only if $G*v$ is an $\mathcal{R}_{d+1}$-circuit.
\end{lemma}

As a generalisation of \added{$\cR_d(K_n)$}, Graver \cite{Gra} introduced an abstract family of matroids which share many of the properties of rigidity matroids.
An \emph{abstract $d$-rigidity matroid} is a matroid $\cM$ on the edge set of the complete graph $K_n$ with closure operator $\cl_\cM$ which satisfies the following two properties.
\begin{enumerate}
    \item[(R1)] If $G_1,G_2 \subseteq K_n$ with $|V(G_1)\cap V(G_2)|\leq d-1$, then $\cl_\cM(G_1\cup G_2)\subseteq K_{V(G_1)}\cup K_{V(G_2)}$.
    \item[(R2)] If $G_1,G_2 \subseteq K_n$ with $|V(G_1)\cap V(G_2)| \geq d$ and $\cl_\cM(G_1) = K_{V(G_1)}$, $\cl_\cM(G_2) = K_{V(G_2)}$, then $\cl_\cM(G_1 \cup G_2) = K_{V(G_1) \cup V(G_2)}$.
\end{enumerate}
These axioms hold when $\cM=\added{\cR_d(K_n)}$ by Lemma~\ref{lem:intbridge} \eqref{it:intbridge:rank} and \eqref{it:intbridge:rig} respectively.
A graph $G\subseteq K_n$ is said to be \emph{$\cM$-rigid} if $\cl_\cM(G) = K_{V(G)}$.
Some results we state for \added{$\cR_d(K_n)$} will generalise straightforwardly to abstract $d$-rigidity matroids.

\subsection{$k$-fold circuits}

We next recall the concepts and results from \cite{JNS} that we will need. 

\begin{definition}
    Let $\cM = (E,r)$ be a matroid and $k \in \NN$ a nonnegative integer.
    A \emph{$k$-fold circuit}  of $\cM$ is a cyclic set $D \subseteq E$ with $r(D) = |D| - k$.
    The \emph{principal partition of $D$} is the partition $\{A_1, \dots, A_\ell\}$ of $D$ where $\{B_i \mid 1 \leq i \leq \ell\}$ is the set of all $(k-1)$-fold circuits of $\cM$ contained in $D$, and $A_i=D \setminus B_i$ for $1\leq i\leq \ell$.
\end{definition}

\added{As with other rigidity matroidal properties of graphs, we say $G = (V,E)$ is a \emph{$k$-fold $\cR_d$-circuit} if $E$ is a $k$-fold circuit in \added{$\cR_d(G)$}, i.e. $r_d(G) = r_d(G - e) = |E| - k$ for all $e \in E(G)$.}

\added{The fact that the principal partition is indeed a partition of $D$ follows from arguments on $\cM^*$ and its lattice of flats; see Remark \ref{rem:lattice+restriction}.
Alternatively, one can see the principal partition forms a partition by defining it in terms of an equivalence relation on $D$.
This is what the next result states, along with a lower bound on the number of parts in any $k$-fold circuit.}
%The next result states that we can equivalently define the principal partition of a $k$-fold circuit in terms of an equivalence relation defined in terms of rank function of $\cM$ and that there is a lower bound on the number of parts in any $k$-fold circuit.

\begin{proposition}[\cite{JNS}]\label{prop:principalpartfork}
Let $\cM=(E,r)$ be a matroid and $D \subseteq E$ be a $k$-fold circuit of $\cM$.
    \begin{enumerate}[label=(\alph*)]
        \item The principal partition of $D$ has at least $k$ parts,
        \item Two elements $e,f \in D$ are contained in the same part of the principal partition if and only if $r(D-e-f) = r(D) -1$.
    \end{enumerate}
\end{proposition}

We say that a $k$-fold circuit is  \emph{trivial}  if its principal partition has $k$ parts,  and {\em non-trivial} otherwise.
Trivial $k$-fold circuits can be characterised as follows.

\begin{lemma}[\cite{JNS}]\label{lem:trivialkcircuit}
    Let $D$ be a trivial $k$-fold circuit in a matroid $\cM$ and $\{A_1,A_2,\ldots,A_k\}$ be its principal partition. Then $A_i$ is a circuit of $\cM$ for all $1\leq i\leq k$ and $\cM|_D=\cM|_{A_1}\oplus \cM|_{A_2}\oplus \ldots \oplus \cM|_{A_k}$.
\end{lemma}

\begin{example}\label{ex:dc}
Figure \ref{fig:2double} provides four examples of double $\cR_2$-circuits.
In both (a) and (b), there are exactly two $\cR_2$-circuits, the two copies of $K_4$ in each case.
The corresponding principal partitions $\{A_1,A_2\}$ are immediate.
As such, they are both trivial double circuits of $\cR_2$.
In (c), both copies of $K_4$ are $\cR_2$-circuits but so is the graph obtained by deleting the edge $v_1v_2$.
Hence the principal partition has 3 parts $\{A_1,A_2,A_3\}$ where $A_3=\{v_1v_2\}$.
In (d), there are 7 distinct $\cR_2$-circuits. These are the unique copy of $K_4$ and the 6 (spanning) $\cR_2$-circuits obtained by deleting exactly one edge from the $K_4$. Hence the principal partition $\{A_1,A_2,\dots,A_7\}$ has $A_1$ equal to the set of 5 edges incident to $u_1$ and $u_2$ and $A_2, A_3,\dots,A_7$ are singleton sets corresponding to the edges in the $K_4$.
\end{example}

\begin{figure}[h]
\begin{center}
\includegraphics[width=\textwidth]{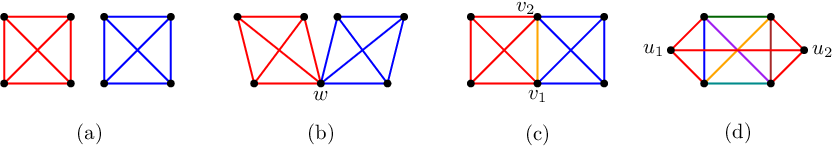}
\end{center}
\caption{Four examples of double $\cR_2$-circuits.}
\label{fig:2double}
\end{figure}

\begin{remark}\label{rem:lattice+restriction}
    Many of the fundamental properties of $k$-fold circuits arise from the observation that $D \subseteq E$ is a $k$-fold circuit of a matroid $\cM$ if and only if $E \setminus D$ is a flat of the dual matroid $\cM^*$ of rank $|E|-r(E) - k$.
    From this, one can deduce that the cyclic sets of $\cM$ form a lattice that is dual to the lattice of flats of $\cM^*$.
    This viewpoint also sheds light on the principal partition of $D$: it is precisely the set of rank one flats of $(\cM|_D)^*$, the dual of the restriction of $\cM$ to $D$.
    For a detailed discussion, see \cite{JNS}.
\end{remark}

Our definition of the $k$-fold circuit property
is motivated by the following theorem.

\begin{theorem}[\cite{JNS}]\label{thm:strong}
    Let $D$ be a $k$-fold circuit in a matroid $\cM=(E,r)$ and $\{A_1, \dots, A_\ell\}$ be the principal partition of $D$. Then 
    \begin{equation}\label{eq:dcpkrev}
   r\left(\bigcap_{i=1}^\ell \cl(D\setminus A_i)\right) \leq \ell-k\,.
\end{equation} 
In addition, \eqref{eq:dcpkrev} holds with equality {if and only if}
$\bigcap_{i=1}^{n-1} \cl(D\setminus A_i)$ and $\cl(D\setminus A_n)$ is a modular pair of flats of $\cM$ for all $2\leq n\leq \ell$.
\end{theorem}

\added{We say a $k$-fold circuit $D$ is \emph{balanced} if \eqref{eq:dcpkrev} holds with equality.} %the right hand side of \eqref{eq:dcpkrev} achieves its maximum possible value.
Our next result shows that the direct sum of balanced $k$-fold circuits remains balanced.

\begin{theorem}[\cite{JNS}]\label{thm:disconalt} 
Suppose $D$ is a disconnected $k$-fold circuit in a matroid $\cM$ and $D_1,D_2,\ldots,D_s$ are its connected components.
If $D_i$ is a balanced $k_i$-fold circuit for all $1\leq i\leq s$, then $D$ is a balanced $k$-fold circuit.
\end{theorem}

\begin{corollary}[\cite{JNS}]\label{cor:trivialstrongalt} For all $k\geq 1$, every trivial $k$-fold circuit of a matroid is balanced.
\end{corollary}

\subsection{Operations on $k$-fold circuits}

There are a number of natural matroid operations that preserve cyclic sets.
The following result demonstrate how $k$-fold circuits and their principal partitions behave under direct sum.

\begin{proposition}[\cite{JNS}]\label{prop:decompose}
    Let $\cM = (E,r)$ be a matroid and $D,D_1,\ldots,D_s \subseteq E$ such that  $\cM|_D = \bigoplus_{i=1}^s\cM|_{D_i}$.
    Then $D$ is a $k$-fold circuit if and only if each $D_i$ is a $k_i$-fold circuit and $k = \sum_{i=1}^s k_i$.
    Moreover, if $D_i$ has principal partition $\cA_i$, then the principal partition of $D$ is $\cA = \bigcup_{i=1}^s \cA_i$.
\end{proposition}

For our applications to $k$-fold circuits in rigidity matroids, we prove analogous results for 2-sum and parallel connection.

\begin{proposition}\label{prop:2sumffold}
    Let $\cM = (E,r)$ be a matroid and $D,D_1,D_2 \subseteq E$ satisfying $e = D_1 \cap D_2$ and $D = (D_1 \cup D_2) - e$ such that $\cM|_D = \cM|_{D_1} \oplus_2  \cM|_{D_2}$  is the 2-sum along $e$.
    %Then:
    \begin{itemize}
        \item[(a)] $D$ is a $k$-fold circuit if and only if each $D_i$ is a $k_i$-fold circuit and $k = k_1+k_2-1$.
        \item[(b)] If $D_i$ is a $k_i$-fold circuit with
    principal partition $\cA_i$ with $e \in A^{i} \in \cA_i$, then the principal partition of $D$ is
    \[
    \cA = (\cA_1 \setminus A^1) \cup (\cA_2 \setminus A^2) \cup \{(A^1 \cup A^2) - e\} \, .
    \]
    \end{itemize}
\end{proposition}

\begin{proof}
For part (a), $D$ is cyclic if and only if $D_i$ is cyclic follows from \eqref{eq:2sum+circuits}.
The claimed values of $k$ and $k_i$ follows by an application of Lemma \ref{lem:2summat} to  $\cM|_D$,  $\cM|_{D_1}$ and $\cM|_{D_2}$ using the fact that $|D|=|D_1|+|D_2|-2$.

Without loss of generality we assume $\cM = \cM|_D$, and write $\cM_i = \cM|_{D_i}$.
\added{As such $\cM$ and $\cM_i$ have no coloops, and so $\cM^*$ and $\cM^*_i$ have no loops.}
By Remark~\ref{rem:lattice+restriction}, the principal partition of $D$ (resp. $D_i$) is the set of rank one flats of $\cM^*$ (resp. $\cM_i^*$).
\added{Let $A \subseteq D$ be a rank one flat of $\cM^*$, i.e. $A \in \cA$.
As duality commutes with 2-sum (\cite[Proposition 7.1.22(i)]{Oxl}), we have that $A$ is a flat of $\cM^* = \cM_1^* \oplus_2 \cM_2^*$.
It follows from the flat description of 2-sums that either $A \cap (D_i - e)$ is a flat of $\cM^*_i$ for both $i=1,2$, or $(A \cup \{e\}) \cap D_i$ is a flat of $\cM^*_i$ for both $i=1,2$.
%Recall that a set $A \subseteq D$ is a flat of $\cM_1^*\oplus_2 \cM_2^*$ if and only if $A \cap (D_i - e)$ is a flat of $\cM^*_i$ for $i=1,2$, or $(A \cup \{e\}) \cap D_i$ is a flat of $\cM^*_i$ for $i=1,2$.
%If $A$ is rank one and in the former case,
%If $A$ is rank one and in the latter case,
If the latter case holds, then $(A \cup \{e\}) \cap D_i$ is a rank one flat in both $\cM_i^*$ for both $i=1,2$, namely $A^i$, and so $A = (A^1 \cup A^2) -e$.
If the former case holds, as $r_{\cM^*}(A) = 1$, it follows without loss of generality that $A \cap (D_1 - e)$ is a rank one flat in $\cM_1^*$ and $A \cap (D_2 - e)$ is a rank zero flat in $\cM_2^*$.
As $\cM_2^*$ is loopless, it follows that $A \subseteq D_1 - e$ and hence $A \in \cA_1 \setminus A^1$.}
%As the rank one flats of $\cM^*_i$ are the parts of the principal partition of $D_i$, 
This gives exactly the description of the principal partition from part (b).
\end{proof}

\begin{proposition}\label{prop:pckfold}
    Let $\cM = (E,r)$ be a matroid and $D,D_1,D_2 \subseteq E$ satisfying $e = D_1 \cap D_2$ and $D = D_1 \cup D_2$ such that $\cM|_D = P(\cM|_{D_1}, \cM|_{D_2})$  is the parallel connection along $e$.
    %Then:
    \begin{itemize}
        \item[(a)] $D$ is a $k$-fold circuit if and only if each $D_i$ is a $k_i$-fold circuit and $k = k_1+k_2$.
        \item[(b)] If $D_i$ is a $k_i$-fold circuit with principal partition $\cA_i$ with $e \in A^{i} \in \cA_i$, then the principal partition of $D$ is
    \[
    \cA = (\cA_1 \setminus A^1) \cup (\cA_2 \setminus A^2) \cup \{A^1  - e, A^2 - e, e\} \, .
    \]
    \end{itemize}
\end{proposition}

\begin{proof}
    Part (a) follows by an application of Lemma \ref{lem:2summat} to  $\cM|_D$,  $\cM|_{D_1}$ and $\cM|_{D_2}$ using the fact that $|D|=|D_1|+|D_2|-1$.

    Again, without loss of generality we assume $\cM = \cM|_D$, and write $\cM_i = \cM|_{D_i}$.
    Similar to Proposition~\ref{prop:2sumffold}, we show part (b) by taking the dual of $\cM$ and determining its rank one flats.
    By \cite[Proposition 7.1.14]{Oxl}, $\cM^* = S(\cM_1^*, \cM_2^*)$ is the series connection of $\cM_1^*$ and $\cM_2^*$.
    The flats of the series connection are precisely $A_1 \cup A_2$ where $A_i$ is a flat of $\cM^*_i$ that does not contain $e$, and $E - (B_1 \Delta B_2)$ where $E_i \setminus B_i$ is a flat in $\cM_i$~\cite[Proposition 7.6.8]{Brylawski}.
    If $A$ is a rank one flat in the former case, it follows that $A \subseteq D_i - e$ a rank one flat in $\cM_i^*$ for some $i=1,2$.
    If $A = E \setminus (B_1 \Delta B_2)$ is a rank one flat in the latter case, then there are precisely three cases how $A$ can arise.
    If $B_1 = E_1 - A^1$ and $B_2 = E_2$, then $A = A^1 - e$; a symmetric argument gives the case where $A = A^2 - e$.
    Finally if $B_1 = E_1$ and $B_2 = E_2$, then $A = E \setminus (E_1 \Delta E_2) = e$ is the final rank one flat.
    As the rank one flats of $\cM^*_i$ are the parts of the principal partition of $D_i$, this gives exactly the description of the principal partition from part (b).
\end{proof}

\section{\boldmath $k$-fold circuits in rigidity matroids}
\label{sec:kfold_in_rigidity_matroids}

In this section we will describe how $k$-fold circuits in rigidity matroids can be constructed using the graphical direct sum, parallel connection and 2-sum operations and obtain sufficient conditions for them to be balanced.
We begin our study by making some simple observations on $k$-fold circuits that generalise properties of circuits in rigidity matroids.

\begin{lemma}\label{lem:maxk}
Let $G=(V,E)$ be a $k$-fold $\cR_d$-circuit.
Then $|E|\leq d|V|-{d+1\choose 2}+k$ with equality if and only if $G$ is $\cR_d$-rigid.
    Moreover, the minimum degree $\delta(G)$ of $G$ satisfies
    \[
    d+1 \leq \delta(G) \leq 2d + \frac{2k - d(d+1)}{|V|} \, .
    \]
\end{lemma}

\begin{proof}
The first claim follows from Lemma~\ref{lem:max} and the definition of $k$-fold circuits.
 
    The lower bound on $\delta(G)$ follows from the well known result that $\cR_d$-circuits have minimum degree at least $d+1$, see for example~\cite[Lemma 2.4]{JJ}, and the fact that every edge of a $k$-fold circuit belongs to an $\cR_d$-circuit. 
    The upper bound now follows from the first part of the lemma since
    \[
    \delta(G) \leq \frac{2|E|}{|V|} \leq 2d + \frac{2}{|V|}\left(k - {d+1 \choose 2}\right) \leq 2d + \frac{2k - d(d+1)}{|V|} \, .
    \]
\end{proof}

\subsection{Constructions}

We will consider graphical versions of the matroid 2-sum and parallel connection operations. 
Let $G_1=(V_1,E_1)$ and $G_2=(V_2,E_2)$ be two graphs with $G_1\cap G_2\cong K_2$ and $E_1\cap E_2=\{e\}$.
We say that:
\begin{enumerate}
    \item $G = (V,E)$ is the {\em graphical parallel connection of $G_1,G_2$ along the edge $e$} if $G=G_1\cup G_2$;
    \item $G = (V,E)$ is the {\em graphical $2$-sum of $G_1,G_2$ along the edge $e$} if $G=(G_1\cup G_2)-e$.
\end{enumerate}
We make no additional assumptions on $e$ (unlike the matroid versions in which $e$ cannot be a loop or a coloop).

The following result of Grasegger et al.~\cite{GGJN} shows that $\cR_d$-circuits are closed under graphical $2$-sums.
It was previously proved for $d=2$ in \cite{BJ} and for $d=3$ in \cite{Tay}. 

\begin{lemma}
\label{lem:2-sum}
Suppose that $G$ is the graphical $2$-sum of two graphs $G_1$ and $G_2$ along an edge $e$. Then $G$ is an $\cR_d$-circuit if and only if $G_1$ and $G_2$ are both $\cR_{d}$-circuits. 
\end{lemma}

We will obtain analogous results for 2-sums and parallel connections of $k$-fold circuits. Our main tool is the following theorem which shows that, for rigidity matroids, the graphical and matroidal operations of both parallel connection and 2-sum commute. The theorem follows easily from results of Servatius and Servatius \cite{Ser}.

\begin{theorem}\label{thm:rigidity+2sum+pc}
    Let $G_1 = (V_1, E_1)$ and $G_2 = (V_2, E_2)$ be two graphs with $E_1 \cap E_2=\{e\}$ such that $e$ is not a coloop in either $\cR_d(G_1)$ or $\cR_d(G_2)$.
    If $G$ is the graphical parallel connection of $G_1$ and $G_2$ along $e$, then
        \[
        \cR_d(G) = P(\cR_d(G_1), \cR_d(G_2)) \, ,
        \]
    Moreover, if $G'$ is the graphical 2-sum of $G_1$ and $G_2$ along $e$, then
        \[
        \cR_d(G') = \cR_d(G_1) \oplus_2 \cR_d(G_2) \, .
        \]
\end{theorem}

\begin{proof}
    The claim for 2-sums is given in \cite[Corollary 1]{Ser} since $e$ can never be a loop in the rigidity matroid of a generic framework.

    For the parallel connection claim, we consider the circuits of $\cR_d(G)$.
    Clearly any circuits of $\cR_d(G_1)$ and $\cR_d(G_2)$ are also circuits of $\cR_d(G)$.
    Moreover, \cite[Theorem 1]{Ser} shows that $(C_1 \cup C_2) -e$ is also a circuit of $\cR_d(G)$ for any $e \in C_i \in \cC(\cR_d(G_i))$.
    This implies that $\cC(P(\cR_d(G_1), \cR_d(G_2))) \subseteq \cC(\cR_d(G))$.

    To show the other containment, let $C \in \cC(\cR_d(G))$.
    If $C \subseteq G_i$ for either $i=1,2$, then this is just a circuit of $\cR_d(G_i)$.
    Hence we write $C = C_1 \cup C_2$ where $C_i = C \cap G_i$ and assume $C_i \neq \emptyset$ or $\{e\}$.
    Note that $e \notin C$, else Lemma~\ref{lem:intbridge}\eqref{it:intbridge:indep} would imply that either $C_1$ or $C_2$ were dependent, contradicting that $C$ is a circuit.
    Hence $C$ can be written as the graphical 2-sum of $C_1 + e$ and $C_2 + e$.
    It follows from Lemma~\ref{lem:2-sum} that this can only occur if $C_1 +e$ and $C_2+e$ were circuits in $\cR_d(G_1)$ and $\cR_d(G_2)$ respectively.
    As such $C \in \cC(P(\cR_d(G_1), \cR_d(G_2))$, giving the designed result.
\end{proof}

Combining Theorem~\ref{thm:rigidity+2sum+pc} with Propositions~\ref{prop:2sumffold} and \ref{prop:pckfold}, we can characterise the behaviour of $k$-fold $\cR_d$-circuits under graphical 2-sums and parallel connection.
These propositions also describe how the principal partitions change, but we will not repeat this in the graphical lemmas.

\begin{lemma}\label{lem:k-sum}
Let $G$ be the graphical 2-sum of two graphs $G_1$ and $G_2$ along an edge $e$.
Suppose that $e$ is not a coloop in either $\cR_d(G_1)$ or $\cR_d(G_2)$.
Then $G$ is a $k$-fold $\cR_d$-circuit if and only if $G_1$ is a $k_1$-fold $\cR_d$-circuit and $G_2$ is a $k_2$-fold $\cR_d$-circuit such that $k_1+k_2=k+1$.
\end{lemma}

\begin{lemma}\label{lem:k-pc}
Let $G$ be the graphical parallel connection of two graphs $G_1$ and $G_2$ along an edge $e$.
Suppose that $e$ is not a coloop in either $\cR_d(G_1)$ or $\cR_d(G_2)$.
Then $G$ is a $k$-fold $\cR_d$-circuit if and only if $G_1$ is a $k_1$-fold $\cR_d$-circuit and $G_2$ is a $k_2$-fold $\cR_d$-circuit such that $k_1+k_2=k$.
\end{lemma}

We cannot remove the hypothesis that $e$ is not a coloop in $\cR_d(G_i)$ for both $i=1,2$  from either Lemma \ref{lem:k-sum} or Lemma \ref{lem:k-pc}.
As an example, Figure~\ref{fig:2sum+counterexample} shows we can obtain a trivial double $\cR_2$-circuit by taking the graphical 2-sum of two graphs which are not $\cR_2$-circuits.
Similar examples can be constructed for parallel connection.

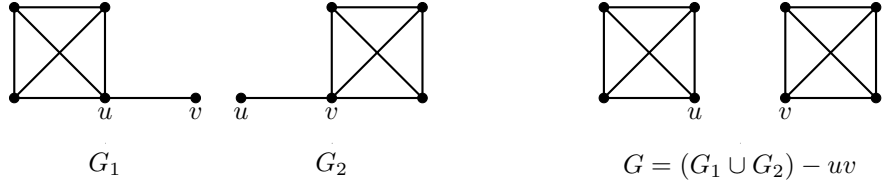
\begin{figure}[h]
\begin{center}
\begin{tikzpicture}[scale=0.6]

\filldraw (0,0) circle (3pt)node[anchor=north]{};
\filldraw (2,0) circle (3pt)node[anchor=north]{$u$};
\filldraw (0,2) circle (3pt)node[anchor=north]{};
\filldraw (2,2) circle (3pt)node[anchor=north]{};
\filldraw (4,0) circle (3pt)node[anchor=north]{$v$};

\filldraw (7,0) circle (3pt)node[anchor=north]{$v$};
\filldraw (9,0) circle (3pt)node[anchor=north]{};
\filldraw (7,2) circle (3pt)node[anchor=north]{};
\filldraw (9,2) circle (3pt)node[anchor=north]{};
\filldraw (5,0) circle (3pt)node[anchor=north]{$u$};

\draw[black,thick]
(0,0) -- (2,0) -- (0,2) -- (2,2) -- (0,0) -- (0,2);

\draw[black,thick]
(4,0) -- (2,0) -- (2,2);

\draw[black,thick]
(9,0) -- (7,2) -- (9,2) -- (7,0) -- (9,0) -- (9,2);

\draw[black,thick]
(5,0) -- (7,0) -- (7,2);

\filldraw (2,-1) circle (0pt)node[anchor=north]{$G_1$};
\filldraw (7,-1) circle (0pt)node[anchor=north]{$G_2$};

%%%%%%

\filldraw (13,0) circle (3pt)node[anchor=north]{};
\filldraw (15,0) circle (3pt)node[anchor=north]{$u$};
\filldraw (13,2) circle (3pt)node[anchor=north]{};
\filldraw (15,2) circle (3pt)node[anchor=north]{};

\filldraw (17,0) circle (3pt)node[anchor=north]{$v$};
\filldraw (19,0) circle (3pt)node[anchor=north]{};
\filldraw (17,2) circle (3pt)node[anchor=north]{};
\filldraw (19,2) circle (3pt)node[anchor=north]{};

\draw[black,thick]
(13,0) -- (15,0) -- (13,2) -- (15,2) -- (13,0) -- (13,2);

\draw[black,thick]
(15,0) -- (15,2);

\draw[black,thick]
(17,0) -- (19,0) -- (17,2) -- (19,2) -- (17,0) -- (17,2);

\draw[black,thick]
(19,0) -- (19,2);

\filldraw (16,-1) circle (0pt)node[anchor=north]{$G = (G_1 \cup G_2) - uv$};
\end{tikzpicture}
\end{center}
\caption{A trivial double $\cR_2$-circuit  arising as a 2-sum where neither graph is a double $\cR_2$-circuit.}
\label{fig:2sum+counterexample}
\end{figure}

\begin{example}\label{ex:db}
    The double banana is a well known example of a flexible $\cR_3$-circuit, obtained as the graphical $2$-sum of two copies of $K_5$ along a common edge $e$. 
    For $d\geq 3$, it can be generalised to the graph $B_{d,d-1}$, defined by letting $B_{d,d-1} = (G_1 \cup G_2)-e$ where $G_i\cong K_{d+2}, G_1 \cap G_2\cong K_{d-1}$ and $e\in E(G_1 \cap G_2)$.
    From \cite[Lemma 11]{GGJN} and its preceding discussion, it follows that $B_{d,d-1}$ is a flexible $\cR_d$-circuit with $r_d(B_{d,d-1})=d(d+5)-{d+1\choose 2}-1$.
    Note that when $d=3$, there is only one choice of edge $e \in E(G_1 \cap G_2)$ and so we recover the double banana $B_{3,2}$.
    
    Define $\overline B_{d,d-1}$ to be the graph obtained from $B_{d,d-1}$ by adding back the edge $e$.
    Then $\cl_d (B_{d,d-1})=\overline B_{d,d-1}$, 
    and hence $\overline B_{d,d-1}$ is a flexible double $\cR_d$-circuit.
    We define the triple banana $B^{(3)}_{d,d-1}$ to be the graphical $2$-sum of $\overline B_{d,d-1}$ and $K_{d+2}$ again along $e$.
    By Lemma~\ref{lem:k-sum}, $B^{(3)}_{d,d-1}$ is a double $\cR_d$-circuit.
    Iterating this process, we get that the $(k+1)$-tuple banana $B^{(k+1)}_{d,d-1}$ is a $k$-fold $\cR_d$-circuit.
    The $d=3$ case is displayed in Figure~\ref{fig:k+banana}.
\end{example}

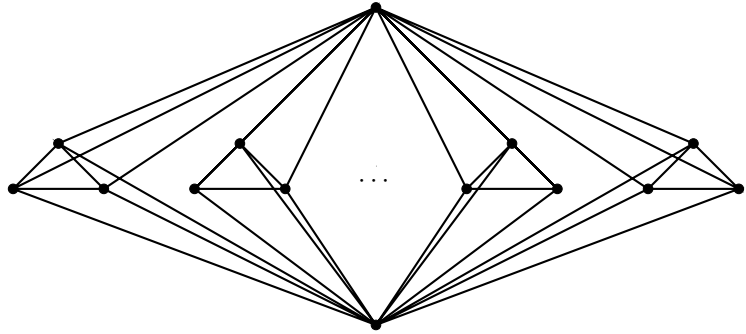
\begin{figure}[h]
\begin{center}
\begin{tikzpicture}[scale=0.6]

\filldraw (8,4) circle (3pt)node[anchor=north]{};
\filldraw (8,-3) circle (3pt)node[anchor=north]{};

\filldraw (0,0) circle (3pt)node[anchor=north]{};
\filldraw (2,0) circle (3pt)node[anchor=north]{};
\filldraw (1,1) circle (3pt)node[anchor=north]{};

\filldraw (4,0) circle (3pt)node[anchor=north]{};
\filldraw (6,0) circle (3pt)node[anchor=north]{};
\filldraw (5,1) circle (3pt)node[anchor=north]{};

\filldraw (8,0.5) circle (0pt)node[anchor=north]{$\cdots$};

\filldraw (10,0) circle (3pt)node[anchor=north]{};
\filldraw (12,0) circle (3pt)node[anchor=north]{};
\filldraw (11,1) circle (3pt)node[anchor=north]{};

\filldraw (14,0) circle (3pt)node[anchor=north]{};
\filldraw (16,0) circle (3pt)node[anchor=north]{};
\filldraw (15,1) circle (3pt)node[anchor=north]{};

\draw[black,thick]
(8,4) -- (0,0) -- (2,0) -- (8,-3) -- (0,0) -- (1,1) -- (8,4) -- (2,0) -- (1,1) -- (8,-3);

\draw[black,thick]
(8,4) -- (4,0) -- (6,0) -- (8,-3) -- (4,0) -- (5,1) -- (8,4) -- (6,0) -- (5,1) -- (8,-3);

\draw[black,thick]
(8,4) -- (10,0) -- (12,0) -- (8,-3) -- (10,0) -- (11,1) -- (8,4) -- (12,0) -- (11,1) -- (8,-3);

\draw[black,thick]
(8,4) -- (14,0) -- (16,0) -- (8,-3) -- (14,0) -- (15,1) -- (8,4) -- (16,0) -- (15,1) -- (8,-3);
\end{tikzpicture}
\end{center}
\caption{The $(k+1)$-tuple banana $B_{3,2}^{(k+1)}$ from Example~\ref{ex:db}, obtained as an iterated graphical 2-sum of $k+1$ copies of $K_5$.
It is a flexible $k$-fold $\cR_3$-circuit.}
\label{fig:k+banana}
\end{figure}

\subsection{Balanced $k$-fold circuits in rigidity matroids}
\label{sec:balancedrigid}

Recall that a matroid $\cM$ satisfies the $k$-fold circuit property if all its $k$-fold circuits are balanced. It
%For rigidity matroids, it 
is easy to construct examples of rigidity matroids \added{$\cR_d(G)$} where the double circuit property fails if we %allow ourselves to 
consider the rigidity matroid of an arbitrary graph \added{$G$}. However, if $\cM=\added{\cR_d(K_n)}$, then the double circuit property holds for both $d=1,2$:
the former was shown by Dress and Lov\'asz~\cite{DL} while the latter was shown by Makai \cite{Mak}.
It was also shown in~\cite{JNS} that \added{$\cR_1(K_n)$} satisfies the $k$-fold circuit property for all $k \geq  1$.
\added{Our next result shows that for fixed $d \geq 4$ and $k \geq 2$, $\cR_d(K_n)$ does not satisfy the $k$-fold circuit property for sufficiently large $n$}.

\begin{theorem}\label{thm:K_67+dcp}
\added{Fix $d\geq 4$ and $k\geq 2$. Then $\cR_d(K_n)$ does not satisfy the $k$-fold circuit property for any $n \geq k(d+2) + 1$.}
 %$\cR_d$ does not satisfy the $k$-fold circuit property for any $d\geq 4$ and any $k\geq 2$.   
\end{theorem}

\begin{proof}
We first show that \added{$\cR_d(K_n)$} does not satisfy the double circuit property \added{when $n \geq 2d+5$} by considering the graph $K_{d+2,d+3}$.
It contains $d+3$ copies of $K_{d+2,d+2}$, which we label $G_1, \dots, G_{d+3}$, each of which is an $\cR_d$-circuit for all $d\geq 3$ and it is $\cR_d$-flexible for all $d\geq 4$, see \cite[Theorem 5.2.1]{GSS}.
We will show that $K_{d+2,d+3}$ is a double circuit by proving that $r_d(K_{d+2,d+3}) = |E(K_{d+2,d+3})| - 2$.
As it is the union of two circuits, we have that $r_d(K_{d+2,d+3}) \leq |E(K_{d+2,d+3})| - 2$.

% % REF NEW PROOF
% \added{For the converse inequality, we first note that $K_{d+2,d+2}$ is a flat in $\cR_d$.
% If this were not the case and $uv \in \cl_d(K_{d+2,d+2})$ for $u,v$ in the same part of the vertex bipartition, then by the symmetry of $\cR_d$ we would have $\cl_d(K_{d+2,d+2}) = K_{2d+4}$, contradicting that $K_{d+2,d+2}$ is $\cR_d$-flexible.
% Add a new vertex $w$ of degree $d+1$ to $K_{d+2,d+2}$ by connecting it to all but one vertex $u$ in a part of the vertex bipartition: this increases the rank by exactly $d+1$ as we can alternatively model this as adding an edge to $K_{d+2,d+2}$ (increasing the rank) and then performing a 1-extension.
% It follows that
% \begin{align*}
% r_d(K_{d+2,d+3}) &\geq r_d(K_{d+2,d+3} - \{uw\}) = r_d(K_{d+2,d+2}) + (d+1) = (d+2)(d+2) - 1 + (d+1) \\
% &= (d+2)(d+3) - 2 = |E(K_{d+2,d+3})| - 2 \, .
% \end{align*}
% Hence $K_{d+2,d+3}$ is a double $\cR_d$-circuit.}

Let $K$ be the graph obtained from $K_{d+2,d+3}$ by deleting two independent edges, say $uv$ and $u'v'$.
We show $K$ is $\cR_d$-independent, implying that $r_d(K_{d+2,d+3}) \geq |E(K_{d+2,d+3})| - 2$ and hence is equal. 
We will make repeated use of Lemma~\ref{lem:01ext} that 1-extension preserves $\cR_d$-independence.

\begin{figure}
    \centering
    \includegraphics[width=0.4\linewidth]{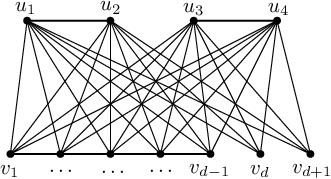}
    \caption{The $\cR_d$-independent graph $G$ from the proof of Theorem~\ref{thm:K_67+dcp}. It is a copy of $K_{4,d+1}$ with $d$ edges added: $u_1u_2, u_3u_4$ and $v_iv_{i+1}$ for $1 \leq i \leq d-2$.}
    \label{fig:R_d+counterexample}
\end{figure}

Consider the graph $G$ in Figure~\ref{fig:R_d+counterexample}.
It is a graph on $d+5$ vertices with $4(d+1) + d = 5d + 4$ edges.
We first note that when $d \geq 4$, we have
\[
|E(G)| = 5d + 4 < \frac{d(d+9)}{2} = d(d+5) - \binom{d+1}{2} \, ,
\]
hence $G$ is $\cR_d$-flexible.
Moreover, \cite[Corollary 2]{GGJN} implies it does not have enough edges to contain a circuit, hence it is $\cR_d$-independent.

We next show that we can obtain $K$ from $G$ via a series of 1-extensions, proving $K$ is also $\cR_d$-independent.
For each $1 \leq i \leq d-2$, we add the degree $d+1$ vertex $u_{i+4}$ by making it adjacent to each $v_j$, and delete the edge $v_iv_{i+1}$.
This gives us a copy of $K_{d+2,d+1}$ with two additional edges.
Finally, we add the degree $d+1$ vertices $v$ and $v'$ by making them adjacent to each $u_i \neq u$ (resp. $u_i \neq u')$ and deleting the edge $u_1u_2$ (resp. $u_3u_4$).
The resulting graph is $K$, demonstrating $K$ is $\cR_d$-independent and hence that $K_{d+2,d+3}$ is a double $\cR_d$-circuit.

We now show that $K_{d+2,d+3}$ is not a balanced double circuit of \added{$\cR_d(K_n)$}.
Recall that $K_{d+2,d+3}$ contains at least $d+3$ $\cR_d$-circuits $G_i \cong K_{d+2,d+2}$.
As the sets $A_i := E(K_{d+2,d+3}) \setminus E(G_i)$ give a partition of the edge set of $K_{d+2,d+3}$, they must form its principal partition.
We prove by contradiction that $G_i$ is a closed set in \added{$\cR_d(K_n)$}.
Suppose $G_i$ is not a closed set, then %\added{there exists $u_1,u_2$ in the same part of the vertex bipartition such that}
without loss of generality, 
adding the edge $u_1u_2$ to $G_i$ does not increase the rank.
Pick any edge $uv \in E(G_i)$ with $u \neq u_1, u_2$.
As $G_i - uv$ is $\cR_d$-independent, it follows that
\added{
\[
r(G_i + u_1u_2 - uv) \leq r(G_i + u_1u_2) = r(G_i) = r(G_i - uv) \leq r(G_i + u_1u_2 - uv)\, ,
\]
and hence $r(G_i) = r(G_i +u_1u_2 - uv)$.
However, we claim that $G_i - uv + u_1u_2$ is $\cR_d$-independent, giving a contradiction.}
We can obtain $G_i - uv + u_1u_2$ from $G$ via a series of 1-extensions as before:
explicitly, we add the degree $d+1$ vertex $u_{m+4}$ by making it adjacent to each $v_j$, and delete the edge $v_mv_{m+1}$ for each $1 \leq m \leq d-2$.
Finally, we add the degree $d+1$ vertex $v$ by making it adjacent to each $u_j \neq u$ and deleting the edge $u_3u_4$.
Hence, we have
\[
r_d\left(\bigcap_{i=1}^{d+3} \cl_d (G_i)\right) = r_d\left(\bigcap_{i=1}^{d+3} G_i\right)=0 < d+1 = d+3-2 \, .
\]
To extend this counterexample to a $k$-fold circuit that is not balanced for $k>2$, let $D$ be $K_{d+2,d+3}$ with $k-2$ edge disjoint copies of the $\cR_d$-circuit $K_{d+2}$.
\added{Observe that $D \subseteq K_n$ if and only if $n \geq k(d+2) + 1$.}
The principal partition of $D$ will have $\ell := d+3 + k-2$ parts by Proposition~\ref{prop:decompose}. 
However, we still have $r_d(\bigcap_{i=1}^{\ell} \cl_d (D \setminus A_i))=0 < d+1 = \ell - k$.
\end{proof}

\noindent
It is an open problem to decide whether the double circuit property holds in \added{$\cR_3(K_n)$}. 

In the remainder of this section, we will obtain a number of sufficient conditions for a particular $k$-fold $\cR_d$-circuit $D$ to be balanced. One of these, Theorem \ref{thm:balanced+rigid+k+circuit} below, implies Makai's result since every $k$-fold circuit in both \added{$\cR_1(K_n)$} and \added{$\cR_2(K_n)$} satisfies the hypotheses of this theorem.

Let $G = (V,E)$ be a $k$-fold $\cR_d$-circuit with principal partition $\cA = \{A_1, \dots, A_\ell\}$.
We can consider the principal partition as a colouring of the edges of $G$.
We say a vertex $x\in V$ is {\em monochromatic} if all edges of $G$ incident with $x$ belong to the same part $A_i$, and that $x$ is \emph{technicolour} if it is not monochromatic.

%The following result is a useful observation when studying $k$-fold $\cR_d$-circuits.
\added{The following results are useful observations when studying $k$-fold $\cR_d$-circuits in terms of their technicolour vertices.}

% \begin{lemma}[\cite{JNS}]\label{lem:tech}
% Let $G=(V,D)$ be a $k$-fold $\cR_d$-circuit with principal partition $\{A_1,A_2,\dots, A_\ell\}$ and let $X$ denote the set of technicolour vertices in $G$. Then $\sum_{i=1}^\ell |V(A_i)|=(\ell-1)|V|+|X|$ and $X$ is the vertex set of $\bigcap_{i=1}^\ell \cl_d(D\setminus A_i)$. 
% \end{lemma}

\added{
\begin{lemma}\label{lem:tech} %[\cite{JNS}]
    Let $G=(V,D)$ be a $k$-fold $\cR_d$-circuit where $k\geq2$ with principal partition $\{A_1,A_2,\dots, A_\ell\}$ and let $X$ denote the set of technicolour vertices in $G$.
    Write $G_i = G[D \setminus A_i]$ for the $(k-1)$-fold $\cR_d$-circuits in $G$.
    Then $\sum_{i=1}^\ell |V(G_i)|=(\ell-1)|V|+|X|$ and $X = \bigcap_{i=1}^\ell V(G_i)$. 
\end{lemma}

\begin{proof}
    Let $Y_i \subseteq V$ be the set of monochromatic vertices whose incident edges are contained in $A_i$.
    Observe that $V(G_i) = V \setminus Y_i$, and that $(X,Y_1, \dots Y_\ell)$ is a partition of $V$.
    It follows that $X = \bigcap_{i=1}^\ell V(G_i)$, and furthermore
    \[
    \sum_{i=1}^\ell |V(G_i)| = \ell|V| - \sum_{i=1}^\ell |Y_i| = (\ell - 1)|V| + |X| \, .
    \]
    % As taking closure does not change the vertex set, we immediately obtain that $X$ is the vertex set of $\bigcap_{i=1}^\ell \cl_d(G_i)$.
\end{proof}}

\added{
\begin{lemma} \label{lem:vertex+separator}
    Let $G$ be a $k$-fold $\cR_d$-circuit where $k \geq 2$, and let $X$ denote the set of technicolour vertices in $G$.
    If $|X| \leq d+1$ then $X$ is a vertex separator of $G$.
\end{lemma}
\begin{proof}
    By Lemma \ref{lem:maxk}, $G$ has at least $d+2$ vertices and so $G - X$ is non-empty.
    %As $k \geq 2$, it follows from... For any two parts $A_i, A_j$ of the principal partition of , we have $V(A_i)\cap V(A_j)\subseteq X$ and hence $X$ is a vertex separator of $G$. 
    %Suppose $|X|<d$. 
    Moreover, observe that every connected component of $G- X$ is monochrome.
    If $G-X$ has exactly one component, its edge set is contained in some part $A_1$ of the principal partition, and then $G_1=G[E\setminus A_1]$ is a $(k-1)$-fold $\cR_d$-circuit with $V(G_1)\subseteq X$. 
    This contradicts Lemma \ref{lem:maxk} that every $(k-1)$-fold $\cR_d$-circuit has at least $d+2$ vertices.
    Hence $G-X$ has at least two components, and so $X$ is a vertex separator.
    %The assumption $|X|<d$ now contradicts the fact that $G$ is $d$-connected by Lemma \ref{lem:rigiddouble}.
\end{proof}

\begin{lemma}\label{lem:technicolour+degree}
    Let $G = (V,E)$ be a $k$-fold $\cR_d$-circuit where $k \geq 2$.
    If $x$ is a technicolour vertex of $G$, then $|N_G(x)| \geq d+2$.
\end{lemma}
\begin{proof}
    Pick some edge $e$ incident to $x$, and let $A_i$ be the part of the principal partition that contains it.
    As $G_i = G[E - A_i]$ is a $(k-1)$-fold $\cR_d$-circuit, it follows from Lemma \ref{lem:maxk} that $|N_{G_i}(x)| \geq d+1$.
    As $N_{G_i}(x) \cup \{e\} \subseteq N_G(x)$, the claim follows.
\end{proof}
}

\subsubsection{\boldmath $k$-fold circuits with few technicolour vertices}

We will show in  Theorem~\ref{thm:atmost2technicolour} below that $k$-fold $\cR_d$-circuits with at most two technicolour vertices are balanced.
We first consider the case where we have at most one technicolour vertex in Proposition~\ref{prop:1technicolour}.
We then prove a classification of connected $k$-fold $\cR_d$-circuits with two technicolour vertices in Proposition~\ref{prop:2technicolour} to assist with the proof of the main result. Throughout, we will make heavy use of the following lemma.

\begin{lemma}[{\cite[Lemma 4.9]{GJ}}]\label{lem:gjcon2sum}
Let $G$ be the graphical 2-sum of $G_1$ and $G_2$ along an edge $e$. Then the following are equivalent:
\begin{enumerate}
    \item $G$ is $\cR_d$-connected;
    \item $G_1$ and $G_2$ are $\cR_d$-connected;
    \item $G+e$ is $\cR_d$-connected. 
\end{enumerate}  
\end{lemma}

\begin{proposition}\label{prop:1technicolour}
Let $G=(V,D)$ be a $k$-fold $\cR_d$-circuit with at most one technicolour vertex.
Then $G$ is a trivial $k$-fold $\cR_d$-circuit.
\end{proposition}

\begin{proof}
We use induction on $k$. The base case when $k=1$ follows since all 1-circuits are trivial. Hence we may suppose that $k\geq 2$. 

Let $X$ be the set of technicolour vertices in $G$.
%For any two parts of the principal partition $A_i, A_j$, we have $V(A_i)\cap V(A_j)\subseteq X$ and hence $X$ is a vertex separator of $G$. 
\added{By Lemma \ref{lem:vertex+separator}, $X$ is a vertex separator of $G$.}
Since 
$|X|\leq 1$, this implies that
$G$ is not $\cR_d$-connected and hence we have $\cR_d(G)=\cR_d(G_1)\oplus\cR_d(G_2)$ for two $k_i$-fold circuits $G_i$ in $\cR_d$ with $k_1+k_2=k$ and $V_1\cap V_2=X$ by Proposition \ref{prop:decompose}.
Let $\cA_i=\{A_1^i,\ldots A_{m_i}^i$\} be the principal partition of $G_i$ and 
$X_i$ be the set of technicolour vertices of $G_i$. Then 
$\cA_1\cup \cA_2$ is the principal partition of $G$ by Proposition \ref{prop:decompose} and 
$X_1\cup X_2\subseteq X$. We may  apply  induction to deduce that each $G_i$ is a trivial $k_i$-fold $\cR_d$-circuit. Lemma \ref{lem:trivialkcircuit} now tells us that $\cR_d(G_i)$ is the direct sum of $k_i$ $\cR_d$-circuits.
Since $\cR_d(G)=\cR_d(G_1)\oplus\cR_d(G_2)$ and $k_1+k_2=k$, $\cR_d(G)$ is the direct sum of $k$ $\cR_d$-circuits. Hence $G$ is a trivial $k$-fold $\cR_d$-circuit.
\end{proof}

Our next result classifies connected $k$-fold $\cR_d$-circuits with only two technicolour vertices, generalising Example \ref{ex:db}.

\begin{proposition}\label{prop:2technicolour}
Let $G=(V,D)$ be an $\mathcal{R}_d$-connected $k$-fold $\cR_d$-circuit for some $k\geq 2$ and let $X=\{u,v\}$ be the set of technicolour vertices in $G$.
\begin{enumerate}[label=(\roman*)] 
    \item If $uv\in D$, then $G$ is the graphical parallel connection of $k$ $\mathcal{R}_d$-circuits along $uv$, \label{technicolor+case+i}
    \item If $uv\notin D$, then $G$ is obtained from the graphical parallel connection of $k+1$ $\mathcal{R}_d$-circuits along $uv$ by deleting $uv$. \label{technicolor+case+ii}
\end{enumerate}
\end{proposition}

\begin{proof}
We prove case~\ref{technicolor+case+i} by induction on $k$, and then use this to deduce case~\ref{technicolor+case+ii}.

Suppose $uv\in D$. 
\added{By Lemma \ref{lem:vertex+separator},} $X$ is a 2-vertex separation of $G$, and we can write $G$ as the graphical parallel connection of two graphs $G_i=(V_i,D_i)$ along $uv$.
Moreover, Lemma~\ref{lem:gjcon2sum} implies $G_1$ and $G_2$ are $\cR_d$-connected, hence $uv$ is not a coloop in either $G_1$ or $G_2$.
We deduce from Lemma~\ref{lem:k-pc} that each $G_i$ is a $k_i$-fold $\cR_d$-circuit where $k = k_1 + k_2$.
For the base case where $k=2$, each $G_i$ must be a (1-fold) $\cR_d$-circuit and so it immediately follows that $G$ is the graphical parallel connection of two $\cR_d$-circuits.

For the case $k \geq 3$, let $\cA_i=\{A_1^i,\ldots, A_{m_i}^i$\} be the principal partition of $G_i$ where $uv \in A_1^i$, and $X_i$ be the set of technicolour vertices of $G_i$.
By Proposition \ref{prop:pckfold}, the principal partition of $G$ is 
\[
\cA = \{uv, A_1^1-uv, A_2^1, \dots, A_{m_1}^1, A_1^2-uv, A_2^2, \dots, A_{m_2}^2\} \, .
\]
By comparing $\cA_i$ with $\cA$, we see that $X_i\subseteq X$.
As $G_i$ is $\cR_d$-connected, Proposition~\ref{prop:1technicolour} implies either $G_i$ is an $\cR_d$-circuit and so $X_i = \emptyset$, or $|X_i| \geq 2$ and hence $X_i = X$.
We may now apply induction to deduce that each $G_i$ is the graphical parallel connection of $k_i$ $\mathcal{R}_d$-circuits along $uv$.
Hence $G$ is the graphical parallel connection of $k$ $\mathcal{R}_d$-circuits along $uv$.

\added{For case~\ref{technicolor+case+ii}, we show $G + uv$ is a $k+1$-fold $\cR_d$-circuit, hence the claim follows from case~\ref{technicolor+case+i}.
%Case~\ref{technicolor+case+ii} where $uv\not\in D$ follows very similarly.
As $X$ is a 2-vertex separator of $G$, we can write $G$ is the graphical 2-sum of two graphs $G_i=(V_i,D_i)$ along $uv$.
Lemma~\ref{lem:gjcon2sum} implies $G + uv$ is $\cR_d$-connected, hence $uv$ is not a coloop in $G+uv$.
It follows that 
\[
r_d(G+uv) = r_d(G) = |E(G)| - k = |E(G+uv)| - (k+1) ,
\]
hence $G+uv$ is a $k+1$-fold $\cR_d$-circuit.}
% By the same deduction as in case \ref{technicolor+case+i}, we deduce that each $G_i$ is an $\cR_d$-connected $k_i$-fold $\cR_d$-circuit where $k + 1 = k_1 + k_2$.
% Let $\cA_i=\{A_1^i,\ldots, A_{m_i}^i$\} be the principal partition of $G_i$ where $uv \in A_1^i$, and $X_i$ be the set of technicolour vertices of $G_i$.
% By Proposition \ref{prop:2sumffold}, the principal partition of $G$ is 
% \[
% \cA = \{(A_1^1 \cup A_1^2)-uv, A_2^1, \dots, A_{m_1}^1, A_2^2, \dots, A_{m_2}^2\} \, .
% \]
% By comparing $\cA_i$ with $\cA$, we see that $X_i\subseteq X$.
% As $G_i$ is $\cR_d$-connected, Proposition~\ref{prop:1technicolour} implies either $G_i$ is an $\cR_d$-circuit and so $X_i = \emptyset$, or $|X_i| \geq 2$ and hence $X_i = X$.
% We may now apply case~\ref{technicolor+case+i} to deduce that each $G_i$ is the graphical parallel connection of $k_i$ $\mathcal{R}_d$-circuits along $uv$.
% Hence $G$ is the graphical parallel connection of $k+1$ $\mathcal{R}_d$-circuits along $uv$ with $uv$ removed.
\end{proof}

\begin{theorem}\label{thm:atmost2technicolour}
Let $G=(V,D)$ be a $k$-fold $\cR_d$-circuit with at most two technicolour vertices.
Then $G$ is a balanced $k$-fold $\cR_d$-circuit.
\end{theorem}

\begin{proof}
Let $X$ be the set of technicolour vertices of $G$.
If $|X|\leq 1$ then Proposition \ref{prop:1technicolour} combined with Corollary~\ref{cor:trivialstrongalt} gives that $G$ is balanced.
Hence we may suppose $X = \{u,v\}$.
Consider the following three cases:

Case 1: $G$ is $\cR_d$-connected, and $uv \in D$.
By Proposition \ref{prop:2technicolour}, $G$ is the graphical parallel connection of $k$ $\cR_d$-circuits $\{C_1, \dots, C_k\}$ along $uv$.
By Lemma~\ref{lem:k-pc}, this implies that $G$ is also the matroidal parallel connection of those circuits.
Iterating Proposition~\ref{prop:pckfold} gives that the principal partition of $G$ is $\{C_1 - uv, \dots, C_k - uv, uv\}$.
Write $G_i = G \setminus (C_i - uv)$ for all $1 \leq i \leq k$ and $G_{k+1} = G - uv$ for the list of all $(k-1)$-fold circuits of $G$.
Note that as $G$ is a $k$-fold circuit, we have $r_d(G) = r_d(G - uv)$, implying that $uv \in \cl_d(G_{k+1})$.
%Coupled with Lemma~\ref{lem:tech}, 
We get that
\begin{align} \label{eq:pc+balanced}
r_d\left(\bigcap_{i=1}^{k+1} \cl_d(G_i)\right) = r_d(\{uv\}) = 1 = (k+1) - k \, ,    
\end{align}
and hence $G$ is balanced.

Case 2: $G$ is $\cR_d$-connected, and $uv \notin D$.
By Proposition \ref{prop:2technicolour}, $G$ is obtained from the graphical parallel connection of $k+1$ $\cR_d$-circuits $\{C_1, \dots, C_{k+1}\}$ along $uv$ by deleting $uv$.
We can also express this as the graphical 2-sum of $G'$ and $C_{k+1}$ along $uv$, where $G'$ is the graphical parallel connection of $\{C_1, \dots, C_k\}$ along $uv$.
By Lemma~\ref{lem:k-sum}, this implies that $G$ is also the matroidal 2-sum of $G'$ and $C_{k+1}$.
As in Case 1, the principal partition of $G'$ is $\{C_1 - uv, \dots, C_k - uv, uv\}$.
As such, applying Proposition~\ref{prop:2sumffold} gives that the principal partition of $G$ is $\{C_1 - uv, \dots, C_{k+1} - uv\}$.
Write $G_i = G \setminus (C_i - uv)$ for $1 \leq i \leq k+1$ for the list of all $(k-1)$-fold circuits of $G$.
Adding $uv$ to $G_i$ gives $G_i + uv$ as the parallel connection of $k$ $\cR_d$-circuits, and hence $G_i + uv$ is a $k$-fold circuit by Proposition~\ref{prop:2technicolour}.
It follows that $r_d(G_i) = r_d(G_i + uv)$ and hence $uv \in \cl_d(G_i)$.
As in Case 1, we get that \eqref{eq:pc+balanced} holds and hence $G$ is balanced.

Case 3: $G$ is not $\cR_d$-connected.
Decompose $\cR_d(G) = \bigoplus_{j = 1}^s \cR_d(G_j)$ into $\cR_d$-connected components.
By Proposition \ref{prop:decompose}, the principal partition of $\cR_d(G_j)$ is precisely the principal partition of $\cR_d(G)$ restricted to $G_j$.
As such, each $\cR_d(G_j)$ has at most two technicolour vertices and so is balanced by the previous cases.
Theorem~\ref{thm:disconalt} now implies that $\cR_d(G)$ is also balanced.
\end{proof}

\subsubsection{\boldmath $k$-fold circuits in which all $(k-1)$-fold circuits are rigid }

We next establish Theorem~\ref{thm:balanced+rigid+k+circuit} which shows that a $k$-fold $\cR_d$-circuit is balanced if all of its $(k-1)$-fold circuits are $\cR_d$-rigid.
To show this, we need the following lemmas relating $\cR_d$-rigid graphs and $k$-fold $\cR_d$-circuits to $d$-connectivity.
Lemma \ref{lem:intbridge}(\ref{it:intbridge:rank}) immediately implies the following well known result.

\begin{lemma}
\label{lem:dcon}
Let $G$ be an $\cR_d$-rigid graph on at least $d+1$ vertices. Then $G$ is $d$-connected.
\end{lemma}

\begin{lemma}\label{lem:dcondouble}
\added{Let $k \geq 2$ and} let $G$ be a $k$-fold $\cR_d$-circuit such that at least two $(k-1)$-fold circuits of $G$ are $\mathcal{R}_d$-rigid.
Then $G$ is $\mathcal{R}_d$-rigid if and only if $G$ is $d$-connected.
\end{lemma}

\begin{proof}
\added{Note that $k$-fold $\cR_d$-circuits have at least $d+2$ vertices by Lemma \ref{lem:maxk}, hence} Lemma \ref{lem:dcon} gives the necessity.
For the sufficiency, by the hypotheses there are two $\cR_d$-rigid $(k-1)$-fold circuits $G_1,G_2$ in $G$ with $G = G_1 \cup G_2$.
Since $G$ is $d$-connected, we have $|V(G_1)\cap V(G_2)|\geq d$.
Hence Lemma \ref{lem:intbridge} implies the result.
\end{proof}

In the case when $k\geq 3$ it would be interesting to give a similar sufficient condition to Lemma \ref{lem:dcondouble} by making assumptions only on the $1$-circuits. However it is not hard to construct examples of a $k$-fold circuit $G$ in $\cR_2$  in which every $1$-circuit is $\cR_2$-rigid and $G$ is $2$-connected but $G$ not $\cR_2$-rigid. One such construction is to take a cycle $C$ of length $k\geq 4$ and replace each edge of the cycle with a copy of $K_4$. It is straightforward to use Lemmas \ref{lem:coning+rank} and \ref{lem:coning+circuit} to extend this example to all $d\geq 2$. 

\begin{lemma}\label{lem:rigiddouble}
Let \added{$d,k \geq 2$} and let $G=(V,D)$ be a $k$-fold $\cR_d$-circuit such that all $(k-1)$-fold circuits of $G$ are $\mathcal{R}_d$-rigid.
Then the following are equivalent:\\
(1) $G$ is $\mathcal{R}_d$-rigid;\\
(2) $G$ is $d$-connected; and\\
(3) the principal partition of $G$ has more than two parts. 
\end{lemma}

\begin{proof}
Lemma \ref{lem:dcondouble} gives $(1)\Leftrightarrow (2)$.

For $(3) \Rightarrow (1)$, choose two sets $A_1, A_2$ from the principal partition of $G$ and put $G_i=G[D\setminus A_i]$ for $i=1,2$.
Note that $A_1 = E(G_2) \setminus E(G_1)$ and $A_2 = E(G_1) \setminus E(G_2)$.
If $|V(G_1)\cap V(G_2)|\geq d$ we may use the same argument as in the proof of Lemma \ref{lem:dcondouble} to deduce that $G$ is $\cR_d$-rigid.
So suppose $|V(G_1)\cap V(G_2)|< d$. As the principal partition has more than two parts, there exists some other part $A_3 \subseteq E(G_1) \cap E(G_2)$ giving rise to another $(k-1)$-fold circuit $G' = G[D \setminus A_3]$ in $G$.
We can write $G'$ as the union of two non-empty subgraphs $G_1' = G_1 \cap G'$ and $G_2' = G_2 \cap G'$.
As these intersect in less that $d$ vertices,
$G'$ is $\cR_d$-flexible (by Lemma \ref{lem:dcon}), contradicting the hypothesis that all the $(k-1)$-fold circuits of $G$ are $\cR_d$-rigid.

We now show $(1)\Rightarrow (3)$. Note that if $k \geq 3$ then $(3)$ is always satisfied by Proposition~\ref{prop:principalpartfork}. Hence we may assume that $k=2$.
Choose two sets $A_1, A_2$ from the principal partition of $G$ and put $G_i=G[D\setminus A_i]$ for $i=1,2$. Note that $G = G_1 \cup G_2$.
If $G_1 \cap G_2$ contains an edge, then the circuit exchange axiom implies the existence of another circuit in $G$ and hence another part in the principal partition, so (3) holds.
Thus we may suppose that $G_1 \cap G_2$ contains no edges. Then $|E| = |E_1| + |E_2|$ where $E_i=D\setminus A_i$ is the edge set of $G_i$.
Since $G,G_1$ and $G_2$ are all $\cR_d$-rigid and $k=2$, Lemmas~\ref{lem:max} and \ref{lem:maxk} imply
$$
d|V| - {d+1 \choose 2} + 2 = |E| = |E_1| + |E_2| = d|V_1| + d|V_2| - 2{d+1 \choose 2} + 2. 
$$
Hence $|V_1| + |V_2| - |V| = \frac{d+1}{2} < d$.
This implies $G_1$ and $G_2$ intersect in less than $d$ vertices, contradicting \added{Lemma \ref{lem:dcon}}. %the fact that $(1)\Rightarrow (2)$ by the previous paragraph.
\end{proof}

\begin{theorem} \label{thm:balanced+rigid+k+circuit}
    Suppose that $G=(V,E)$ is a $k$-fold $\cR_d$-circuit such that every $(k-1)$-fold circuit in $G$ is $\mathcal{R}_d$-rigid.
    Then $G$ is balanced.
\end{theorem}

\begin{proof}
Let $\{A_1,\dots,A_\ell\}$ be the principal partition of $G$ and $X$ be the set of technicolour vertices.
Moreover, we let $G_i=G[D\setminus A_i]$ for $1\leq i \leq \ell$ be the $(k-1)$-fold circuits of $G$ and put $G_i=(V_i,E_i)$.
By the hypothesis that every $(k-1)$-fold circuit is $\cR_d$-rigid, we have $|E_i|=d|V_i|-{d+1 \choose 2}+k-1$, $\cl_d(G_i)$ is complete and $r_d(G_i)=d|V_i|-{d+1 \choose 2}$.
Combined with Lemma \ref{lem:tech}, this implies it will suffice to prove that  
\begin{equation} \label{eq:KX+tight}
r_d\left(\bigcap_{i=1}^\ell \cl_d(G_i)\right) = r_d(K_{X}) = \added{
\begin{cases}
\binom{|X|}{2} & |X| < d \\ d|X|-{d+1\choose 2} & |X| \geq d
\end{cases}
} %\min\left({|X| \choose 2}, d|X|-{d+1\choose 2}\right) \, .
\end{equation}
\added{is equal to $\ell-k$.}

Note that if the principal partition of $G$ has exactly two parts, then $G$ must be a trivial double circuit, 
and the conclusion follows from 
Corollary \ref{cor:trivialstrongalt}.
Hence we may assume $\ell > 2$.
Lemma \ref{lem:rigiddouble} now implies that $G$ is $\cR_d$-rigid, and Lemma \ref{lem:maxk} now gives $|E|=d|V|-{d+1\choose 2}+k$.
\added{Furthermore, Lemma \ref{lem:vertex+separator} gives that $X$ is a vertex separator if $|X| < d$, contradicting that $G$ is $d$-connected by Lemma \ref{lem:rigiddouble}.}
%Note that minimum on the right hand side of \eqref{eq:KX+tight} is $d|X| - {d+1 \choose 2}$ when $|X| \geq d$. We will show this is always the case.
% Suppose $|X|<d$. Observe that every connected component of $G- X$ is 
% monochrome. 
% If $G- X$ had exactly one component, its edge set would be contained in some part $A_1$, and then $G_1=G[E\setminus A_1]$ would be a $(k-1)$-fold $\cR_d$-circuit with
% $V(G_1)\subseteq X$. This is impossible since $|X|<d$.
% Hence $G- X$ has at least two components. The assumption $|X|<d$ now contradicts the fact that $G$ is $d$-connected by Lemma \ref{lem:rigiddouble}.
Hence $|X| \geq d$ and the last equality of \eqref{eq:KX+tight} can be simplified to
\begin{equation} \label{eq:full+rank}
 r_d(K_{X}) = d|X|-{d+1\choose 2} \, .
\end{equation} 

\added{As the principal partition is indeed a partition,} %Proposition \ref{prop:principalpartfork} implies that 
every edge of $G$ is contained in exactly $\ell-1$ $(k-1)$-fold circuits of $G$ and hence $\sum_{i=1}^\ell |E_i|=(\ell-1)|E|$. Thus
\begin{eqnarray*}
d \sum_{i=1}^\ell|V_i|-\ell({d+1 \choose 2}-k+1) &=&  \sum_{i=1}^\ell |E_i| \\
&=& (\ell-1)|E| \\
&=& (\ell-1)(d|V|-{d+1\choose 2}+k) \\
&=& d(\ell-1)|V|-(\ell-1)({d+1 \choose 2}-k).
\end{eqnarray*}
Therefore
\begin{eqnarray*} 
d \sum_{i=1}^\ell|V_i| &=& d(\ell-1)|V|-(\ell-1)({d+1\choose 2}-k)+\ell({d+1\choose 2}-k+1) \\ &=& 
d(\ell-1)|V|+{d+1\choose 2}+\ell-k.
\end{eqnarray*}
Combining the previous equations with Lemma \ref{lem:tech}, we obtain
\begin{eqnarray*} 
d(\ell-1)|V| + d|X| &=& d \sum_{i=1}^\ell|V_i|
=d(\ell-1)|V|+{d+1\choose 2}+\ell-k,
\end{eqnarray*}
and hence
\begin{eqnarray*}
d|X| - {d+1\choose 2} = \ell-k.
\end{eqnarray*}
This gives \eqref{eq:full+rank} and completes the proof of the theorem.
\end{proof}

\begin{remark}
 A fairly straightforward generalisation of the proof shows that Theorem~\ref{thm:balanced+rigid+k+circuit} holds for every abstract $d$-rigidity matroid.
\end{remark}

\subsection{Matroid matchings}\label{sec:matroid+matching}

Let $\cM=(E,r)$ be a matroid on ground set $E$ with rank function $r$.
Let $\cH \subseteq {E\choose 2}$ be a set of (not necessarily pairwise disjoint) pairs of $E$.
A set of pairs $H \subseteq \cH$ is said to be a \emph{matroid matching} of $\cH$ with respect to $\cM$ if $r(\bigcup_{p \in H} p) = 2|H|$.
The \emph{matroid matching problem} is to compute a matroid matching of $\cH$ of maximum size, the size of which is denoted by $\nu(\cH)$.

Lov\'asz's initial interest in double circuits was motivated by applications to the matroid matching problem.
Dress and Lov\'{a}sz \cite{DL} gave a min-max inequality for $\nu(\cH)$ for all matroids $\cM$, and proved that it holds with equality whenever $\cM$ has the double
circuit property.

\begin{theorem}[\cite{DL}]\label{thm:matroid+matching}
Let $\cM=(E,r)$ be a matroid and $\cH \subseteq {E\choose 2}$. Then
\begin{equation} \label{eq:matching+minmax}
    \nu(\cH) \leq \min \left(r(Z)+\sum_{i=1}^t \left \lfloor \frac{r(Z \cup \bigcup_{p \in \cH_i} p) - r(Z)}{2}\right \rfloor \right)
\end{equation}
where the minimum is taken over all flats $Z \subseteq E$ of $\cM$ and for all partitions $\pi = (\cH_1,\cH_2,\dots, \cH_t)$ of $\cH$.
Moreover, if $\cM$ has the double circuit property then \eqref{eq:matching+minmax} holds with equality.
\end{theorem}

Following \cite{DL}, we say that a matroid $\cM$ has the \emph{matroid matching property} if   equality holds in \eqref{eq:matching+minmax} for all $\cH \subseteq {E\choose 2}$.
Hence the second part of Theorem~\ref{thm:matroid+matching} can be rephrased as $\cM$ has the matroid matching property if it has the double circuit property.
In particular Makai's result implies that $\cR_2$ has the matroid matching property.
However the situation is different for $\cR_d$ with $d\geq 3$. 
We saw in Theorem~\ref{thm:K_67+dcp} that $\cR_d$ does not satisfy the double circuit property for any $d\geq 4$.
We next prove the stronger statement that $\cR_d$ does not have the matroid matching property for any even $d\geq 4$.

\begin{proposition}\label{prop:matroid+matching}
$\cR_{2m}(K_n)$ does not have the matroid matching property for all $m \geq 2$ and $n\geq 4m+5$.
\end{proposition}

\begin{proof}
Consider a double circuit $G\cong K_{2m+2,2m+3}$ in $\cR_{2m}(K_n)$ with vertex partition $X=\{x_1,\ldots,x_{2m+2}\}$ and $Y=\{y_1,\ldots,y_{2m+3}\}$. Let $\cH$ be the set of pairs
\[
\cH = \left\{(x_{2i-1}y_j,x_{2i}y_j) \: \Big| \: 1 \leq i \leq m+1 \, , \, 1 \leq j \leq 2m+3 \right\} \, .
\]
Observe that $\cH$ is a partition of $E(G)$ into pairs.
If we remove any pair of edges of  $\cH$ from $G$, the resulting graph
%set of edges 
contains a copy of the $\cR_{2m}$-circuit $K_{2m+2,2m+2}$.
Removing two pairs $\{(x_{2i-1}y_j,x_{2i}y_j), (x_{2i'-1}y_{j'},x_{2i'}y_{j'})\}$ from $G$ with $j \neq j'$ destroys all copies of $K_{2m+2,2m+2}$ in $G$, and hence the resulting graph is $\cR_{2m}$-independent. This implies that $\nu(\cH) = (m+1)(2m+3) -2$.

Let $\alpha(Z, \pi)$ denote the evaluation \added{of the argument} of the right hand side of \eqref{eq:matching+minmax} for a flat $Z$ of $\cR_{2m}(K_n)$ and a partition $\pi$ of $\cH$.
We will show that $\alpha(Z, \pi) > (m+1)(2m+3) -2$ for all choices of $Z$ and $\pi$.

Suppose, for a contradiction, that
\begin{equation}\label{eq:m1}
\alpha(Z, \pi) = (m+1)(2m+3) -2  \mbox{ for some choice of $Z$ and $\pi$.}
\end{equation}
Let $\pi = (\cH_1, \dots, \cH_t)$.
To simplify notation, we put $H_i = \bigcup_{p \in \cH_i} p$ for all $1\leq i\leq t$. 

We first consider the case when there exists an edge $e\in Z\setminus E(G)$.
Choose
$(f,g)\in \cH$ and put $\cH'=\cH+(e,f)$ and $\pi'=\pi+\{(e,f)\}$.
By symmetry we may assume that $e,f$ are not incident with $y_1$. It is straightforward to check that $E(G)-\{x_1y_1,x_2y_1,g\}$ is independent in $\cR_{2m}$. This implies that  $\cH'\setminus \{(x_1y_1,x_2y_1),(f,g)\}$ is a matching in $\cR_{2m}$
and hence $\nu(\cH')\geq (m+1)(2m+3) - 1$. On the other hand 
$$\alpha(Z,\pi')= \alpha(Z,\pi)+\left\lfloor \frac{r_{2m}(Z+f) - r_{2m}(Z)}{2} \right\rfloor=\alpha(Z,\pi)=(m+1)(2m+3) -2.$$ 
This contradicts Theorem \ref{thm:matroid+matching}.

Hence $Z\subseteq  E(G)$. Since $r_{2m}(Z)\leq (m+1)(2m+3) -2$ by \eqref{eq:m1} and every $\cR_{2m}$-circuit in $G$ is a copy of $K_{2m+2,2m+2}$,  $Z$ is $\cR_{2m}$-independent.
Consider the following three cases.
\\[1mm]
Case 1: $H_i \cup Z$ is $\cR_{2m}$-independent for all $1 \leq i \leq t$.
This contradicts \eqref{eq:m1} since it gives
\begin{equation} \label{eq:minmax+all+indep}
\alpha(Z, \pi) \geq \sum_{i=1}^t \left\lfloor \frac{r_{2m}(H_i \cup Z) - r_{2m}(Z)}{2}\right \rfloor = \sum_{i=1}^t \left\lfloor \frac{|H_i|}{2}\right \rfloor = (m+1)(2m+3) \, .
\end{equation}
Case 2: exactly one of the sets $H_i\cup Z$ is $\cR_{2m}$-dependent, say $H_1 \cup Z$.
Since $K_{2m+2,2m+3}$ is a double circuit and $Z\subseteq E(G)$, we have $r_{2m}(H_1 \cup Z) \geq |H_1 \cup Z| - 2$. This again contradicts \eqref{eq:m1} since it gives
\begin{equation}\label{eq:minmax+1+dep}
\alpha(Z, \pi) \geq \sum_{i=1}^t \left\lfloor \frac{r_{2m}(H_i \cup Z) - r_{2m}(Z)}{2}\right \rfloor \geq  \left\lfloor \frac{|H_1|}{2}\right \rfloor - 1 + \sum_{i=2}^t \left\lfloor \frac{|H_i|}{2}\right \rfloor = (m+1)(2m+3) -1 \, .
\end{equation}
Case 3: at least two of the sets  $H_i\cup Z$ are $\cR_{2m}$-dependent, say $H_1 \cup Z$ and $H_2 \cup Z$.
This implies that both $H_1 \cup Z$ and $H_2 \cup Z$ contain a  copy of $K_{2m+2,2m+2}$, and so have $r_{2m}(H_i \cup Z) \geq (2m+2)^2 -1$ for $i = 1,2$.
This again contradicts \eqref{eq:m1} since it gives
\[
\alpha(Z,\pi) \geq |Z| + \sum_{i=1,2} \left\lfloor\frac{r_{2m}(H_i \cup Z) - r_{2m}(Z)}{2} \right\rfloor \geq |Z| + 2\left(\frac{(2m+2)^2 -1 - |Z| -1}{2}\right) \geq (2m+2)^2 -2 \, .
\]
\end{proof}

Our proof of Proposition~\ref{prop:matroid+matching} does not extend to $d$ odd, as we require $X$ to have even cardinality in the construction of $\cH$.
We suspect, however, that the matroid matching property fails for \added{$\cR_d(K_n)$} whenever $d\geq 4$ \added{and $n$ large}.

\begin{conjecture}
    $\cR_{2k+1}(K_n)$ does not have the matroid matching property for all $k \geq 2$ and sufficiently large $n$.
\end{conjecture}

It remains unclear whether \added{$\cR_3(K_n)$} has the double circuit property or the matroid matching property.

\added{\section{Applications to rigidity}}
\label{sec:coning}

\added{We will use $k$-fold circuits to derive some results in rigidity theory, most notably to analyse the coning operation.}
%\textcolor{red}{We will derive some results in rigidity theory.}

\subsection{The coning operation in rigidity matroids}

\added{We will} show that the cone $G*v$ of a graph $G$ is a $k$-fold $\cR_{d+1}$-circuit if and only if $G$ is a $k$-fold $\cR_d$-circuit. We also obtain  
results on how the principal partitions of $G$ and $G*v$ are related.
We then apply these results to obtain further results on the $d$-dimensional rigidity matroid. 

\begin{lemma}\label{lem:coning-k-circuits-weak}
Let $G*v$ be the cone of a graph $G$ over a new vertex $v$, and $Y$ the set of monochromatic vertices of $G$. Then
\begin{enumerate}[label=(\alph*)]
    \item $G$ is a $k$-fold $\cR_d$-circuit if and only if $G*v$ is a $k$-fold $\mathcal{R}_{d+1}$-circuit.
    \item In addition, if $G$ is a $k$-fold $\cR_d$-circuit with principal partition $\cA$ and $A_i\in \cA$,
then
$$
A_i' = A_i \cup \{uv \mid u\in Y\cap V(G[A_i])\}
$$
is a part of the principal partition of $G*v$ in $\mathcal{R}_{d+1}$.
\end{enumerate}
\end{lemma}

\begin{proof}
Write $G = (V,D)$ and $G*v = (V \cup \{v\}, D')$ where $D' = D \cup \{uv \mid u \in V\}$.
We first note that Lemma~\ref{lem:coning+circuit} immediately gives that $G$ is $\cR_d$-cyclic if and only if $G*v$ is $\cR_{d+1}$-cyclic.
By Lemma~\ref{lem:coning+rank}, we also have
\begin{equation}\label{eq:DD}
|D'|=|D|+|V| \mbox{ and } r_{d+1}(G*v)=r_{d}(G)+|V| \, .
\end{equation}
In particular, we have $r_d(G) = |D| - k$ if and only if $r_{d+1}(G) = |D'| - k$, proving (a).

We next verify (b). Since $A_i\in \cA$, it follows that $G[D\setminus A_i]$ is a $(k-1)$-fold $\cR_d$-circuit.
Hence $G[D\setminus A_i]*v$ is a $(k-1)$-fold $\cR_{d+1}$-circuit by (a). This in turn implies that 
$A_i'=D'\setminus E(G[D\setminus A_i]*v)$ is a part of the principal partition of $G*v$ in $\mathcal{R}_{d+1}$.
\end{proof}

Lemma~\ref{lem:coning-k-circuits-weak} gives us partial information about the principal partition of $G*v$.
The restriction of the principal partition of $G*v$ to $E(G)$ is equal to the principal partition of $G$, and each part $A_i$ of the principal partition of $G$ can be extended to a  part of the principal partition of $G*v$ by adding the edges $uv$ for which $u$ is a monochrome vertex in $G[A_i]$.
We can also deduce that if $u$ is technicolour then $uv$ is added to a new class in the principal partition of $G*v$. It remains to determine how the principal partition of $G*v$ partitions the edges from $v$ to the technicolour vertices of $G$.  
We need to introduce some additional geometric notions to address this problem.

Given a framework $(G,p)$, its rigidity matroid $\cR(G,p)$ is the row matroid of the rigidity matrix $R(G,p)$.
An \emph{equilibrium stress} of $(G,p)$ is a vector $\omega$ in the cokernel of $R(G,p)$.
Equivalently $\omega\in \mathbb{R}^{|E|}$ is an equilibrium stress if, for all $v\in V$, 
\begin{equation}\label{eq:1}
 \sum_{u\in N(v)}\omega_{vu}(p(v)-p(u))=0.   
\end{equation}
Note that $G$ is an $\cR_d$-circuit if and only if, for any generic $p:V\to \RR^d$, there is a unique (up to scaling) equilibrium stress of $(G,p)$ which is non-zero on every edge of $G$.

We can quantify dependencies in $\cR(G,p)$ via equilibrium stresses of $(G,p)$ as follows.
Given an equilibrium stress $\omega$ of $(G,p)$, define the \emph{support} of $\omega$ as 
\[
\supp(\omega) = \{e \in E(G) \mid \omega(e) \neq 0 \} \, .
\]
A subgraph $H \subseteq G$ has $\rank R(H,p) < |E(H)|$ if and only if there exists a linear dependence between the rows of $R(H,p)$ indexed by
the edges of $H$, or equivalently there exists a non-zero equilibrium stress $\omega$ of $(G,p)$ such that $\supp(\omega) \subseteq E(H)$.
Similarly, an edge $e$ of $G$ satisfies $\rank R(G,p) = \rank R(G-e,p)$ if and only if there exists an equilibrium stress $\omega$ of $(G,p)$ such that $e \in \supp(\omega)$.
Combining these facts, $E(H)$ is a cyclic set of $\cR(G,p)$ if and only if there exists an equilibrium stress $\omega_H$ with $\supp(\omega_H) = E(H)$.
Moreover, the circuits of $\cR(G,p)$ are in 1:1 correspondence with the \added{non-zero} equilibrium stresses of minimal support, \added{up to scaling}.

The rank function of the rigidity matroid of any $d$-dimensional framework $(G,p)$ is maximised whenever $p$ is generic and hence 
\[
\rank R(H,p) \leq r_d(H) \quad \mbox{ for all } H \subseteq G \, .
\]
In particular, if $H$ is an $\cR_d$-circuit then $E(H)$ is dependent in $\cR(G,p)$ for all realisations $p$ of $G$ in $\RR^d$.

Our final ingredient is a geometric version of Lemma~\ref{lem:coning+rank} from \cite{Wcone}, see also \added{\cite[Theorem 5.2]{GGJ}}. 
% \ben{\cite{GGJ} cite `Coning, Symmetry and Spherical Frameworks' - Schulze, Whitely. Cite that instead?}
% \tony{was this a ref comment? The lemma is precisely as stated in GGJ, and they were the first to publish a proof of circuits transferring by coning so it's a good ref to give, I don't remember what is in SW and not in Wcone that isn't about the symmetric case and hence off-topic for us?}
We recall that a hyperplane $\cH \subseteq \RR^d$ \emph{separates} two points $p,q \in \RR^d$ if and only if the convex line segment between $p$ and $q$ intersects $\cH$.

\begin{lemma}\label{lem:geometriccone}
Let $G*v$ be the cone of a graph $G$ with new vertex $v$. Let $(G*v,p)$ be a framework in $\mathbb{R}^{d+1}$ and let $\cH$ be a hyperplane in $\mathbb{R}^{d+1}$ which separates $p(V)$ from $p(v)$.
Let $(G,p_\cH)$ be the framework in $\mathbb{R}^{d}$ obtained from $(G,p|_G)$ by projecting $p(u)$ onto the hyperplane $\cH$ along the line through $p(u), p(v)$ for all $u\in V(G)$, and then identifying $\cH$ with $\mathbb{R}^d$.
Then $\mbox{rank } R(G*v,p)=\mbox{rank } R(G,p_\cH)+|V(G)|$.
\end{lemma}

Our main result is the following.

\begin{theorem}\label{thm:coning-k-circuits-strong}
Let $G$ be a $k$-fold $\mathcal{R}_d$-circuit with principal partition $\cA$, and $u,x$ be two technicolour vertices of $G$. Then $vu,vx$ belong to different parts of the principal partition of the $k$-fold $\mathcal{R}_{d+1}$-circuit $G*v$ whenever
\begin{enumerate}[label=(\alph*)]
    \item there exists $z\in N_G(u)\cap N_G(x)$ such that $uz$ and $xz$ belong to different parts of $\cA$, or
    \item $N_G(u)-x\neq N_G(x)-u$.
\end{enumerate}
\end{theorem}

\begin{proof}
Throughout the proof, we will make frequent use of the fact that matrix rank is lower semicontinuous, i.e.~sufficiently small perturbations of the entries in a matrix cannot decrease its rank.
In particular, if $\rank R(G,p)=r_d(G)$ for some realisation $p$ of $G$ in $\RR^d$  then $\rank R(G,q)=r_d(G)$ for any $q$ within a sufficiently small open neighbourhood of $p$.

\added{Let $\cA'$ denote the principal partition of $G*v$ in $\cR_{d+1}$.}
Consider a framework $(G*v,p)$ where $p$ is chosen to map every vertex of $G$ generically into the hyperplane $\cH = \{x \in \RR^{n+1} \mid x_{d+1}=0\}$ consisting of all points with last coordinate zero, and $p(v)$ to be a generic point in $\mathbb{R}^{d+1}$.
Then $\rank R(G,p|_G)=r_d(G)$ and the hypothesis that $G$ is a $k$-fold $\cR_d$-circuit combined with Lemmas~\ref{lem:coning+rank} and \ref{lem:geometriccone} gives
\begin{align*}
\rank R(G*v,p) &= \mbox{rank } R(G,p|_G) + |V| = r_{d}(G) + |V| = r_{d+1}(G*v)
= |E(G*v)| - k \, .
\end{align*}
It follows that $\cR(G*v,p)$ contains a unique $k$-fold circuit, which we can construct from $E(G*v)$ by removing all its $\cR(G*v,p)$-bridges.
Moreover, since $G$ is a $k$-fold $\cR_d$-circuit and $(G, p|_G)$ is a generic realisation of $G$ in $\mathcal{H}$, $E(G)$ is the unique $k$-fold circuit in \added{$\cR(G*v,p)$}.
This implies, in particular, 
that there exists an equilibrium stress $\omega$ of $(G*v,p)$ in which $\omega(e) \neq 0$ for all $e\in E(G)$.

Choose $z\in N_G(u)-x$.
We next construct a new framework $(G*v,q)$ by perturbing only the $(d+1)$-th coordinate of $p(z)$ within a sufficiently small open neighbourhood so that $\rank R(G*v,q)=\rank R(G*v,p)$. 
Then $\rank R(G*v,q) =r_{d+1}(G*v)= |E(G*v)| - k$ so $\cR(G*v,q)$ contains a unique $k$-fold circuit $F$. 
We can  obtain $F$ from $E(G*v)$ by removing all the $\cR(G*v,q)$-bridges.

We first observe that $E(G) \subseteq F$ for $q$ sufficiently close to $p$.
This follows since the equilibrium stress $\omega$ of $(G*v,p)$ for which $\omega(e) \neq 0$ for all $e\in E(G)$ gives rise to a stress $w_{p'}$ for all realisations $(G*v,p')$ sufficiently close to $p$, such that $\omega_{p'}$ changes continuously as a function of $p'$ and $\omega_p=\omega$. \added{(This in turn follows from the fact that the rank of the rigidity matrix remains constant in a neighbourhood $U$ of $p$. So a basis for
the space of equilibrium stresses of $(G,p')$ can be written so that the coordinates of each basis element are rational functions of $p'$ for any $p'\in U$).}
%\added{see the proof of \cite[Theorem 5]{CW} for a more detailed proof of this claim.}
Hence $\omega_q(e)\neq 0$ for all $e\in E(G)$ when $q$ is sufficiently close to $p$ and we have  $E(G)\subseteq F$.

We next show that $uv \in F$. To see this, note that 
$\omega_q(uz)\neq 0$ by the preceding paragraph. Since $q(N_G(u)+u-z)\subseteq \cH$, the equilibrium stress condition \eqref{eq:1} guarantees that $\omega_q(uv)\neq 0$.
Hence  $uv$ is not a bridge in $\cR(G*v,q)$, and so $uv \in F$.

Let $\cA_F$ be the principal partition of $F$ in $\cR(G*v,q)$.
We will show \added{that} there exists some $t \in N_G(x)$ such that $uv$ and $tx$ are in different parts of $\cA_F$.
We first observe that $uv$ and $uz$ are in the same part of $\cA_F$, as $v$ and $z$ are the only vertices adjacent to $u$ that have a non-zero $(d+1)$-th coordinate in $(G*v,q)$ and hence every circuit of $\mathcal{R}(G*v,q)$ which contains one of $uv,uz$ must contain both of them.
As $x$ is technicolour in 
$G$,
we may pick some edge $tx \in E(G)$ in a different part to $uz$ in the principal partition $\cA$ of $G$. Then  $r_d(G - uz - tx) = r_d(G)$, and
Lemma \ref{lem:geometriccone} gives
\begin{align*}
\rank R(G*v - uz - tx, q) &\geq \rank R(G*v - uz - tx, p) = \rank R(G - uz - tx, p|_G) + |V| \\
&= r_d(G- uz -tx) + |V| = r_{d}(G) + |V| = r_{d+1}(G*v) \, .
\end{align*}
As $\rank R(G*v - uz - tx, q)$ and $\rank R(G*v, q)$ are both at most $r_{d+1}(G*v)$, they are both equal to $r_{d+1}(G*v)$.
Moreover, restricting $\cR(G*v,q)$ to $F$ is equivalent to removing bridges. Hence,
if we put  $H:=(G*v)[F]$, we have
\[
\rank R(H - uz - tx, q) = \rank R(H, q) \, .
\]
This implies that $uz$ and $tx$ are in different parts of the principal partition $\cA_F$, and hence $tx$ is also in a different part to $uv$. 
Thus there exists an equilibrium stress $\hat \omega$ of $(G*v - uv,q)$ such that $\hat \omega(tx) \neq 0$.

Suppose that condition (a) holds. In this case we may assume that the vertex $z$ has been chosen so that $z\in N_G(u)
\cap N_G(x)$ and $uz$ and $xz$ are in different parts of $\cA$, and we may take $t = z$. Then $\hat \omega(xz)\neq 0$ by the preceding paragraph. Since $q(N_G(x)+x-z)\subseteq \cH$, the equilibrium stress condition \eqref{eq:1} guarantees that $\hat\omega(vx)\neq 0$.
Hence  $vx$ is not a bridge in $\cR(G*v-uv,q)$.
This gives 
\[
r_{d+1}(G*v)=\mbox{rank }R(G*v,q)=\mbox{rank }R(G*v - uv - vx,q)=r_{d+1}(G*v-uv-vx).
\] 
Hence $uv,vx$ are in different classes of $\cA'$, as required.

Suppose that condition (b) holds. In this case we can modify our choice of $z$ so that $z\in N_G(u)\setminus N_G(x)$.
Then $t\neq z$ since $z$ is not a neighbour of $x$. We obtain a new framework $(G*v -uv, \tilde{q})$ from $(G*v -uv, \added{q})$ by perturbing only the $(d+1)$-th coordinate of $q(t)$ within a sufficiently small open neighbourhood so that $\rank R(G*v - uv,\tilde q)=\rank R(G*v - uv, q)=r_{d+1}(G*v)$. 
Since $\hat \omega(tx) \neq 0$, $(G*v - uv,\tilde q)$ will have an equilibrium stress $\hat \omega'$ with $\hat \omega'(tx) \neq 0$ for $\tilde q$ sufficiently close to $q$.
Since $\tilde q(N_G(x)+x-t)\subseteq \cH$, the equilibrium stress condition \eqref{eq:1} guarantees that $\hat\omega'(xv)\neq 0$.
Since $\hat \omega'$ is an equilibrium stress of $(G*v - uv,\tilde q)$ we have
$\mbox{rank }R(G*v -uv,\tilde q) =  \mbox{rank }R(G*v -uv-vx,\tilde q).$
Hence 
\[
r_{d+1}(G*v)=\mbox{rank }R(G*v,\tilde q)=\mbox{rank }R(G*v - uv - vx,\tilde q)=r_{d+1}(G*v-uv-vx)
\] 
and therefore $uv,vx$ are in different classes of $\cA'$ as required.
\end{proof}

\begin{figure}
    \centering
    \includegraphics[width=\linewidth]{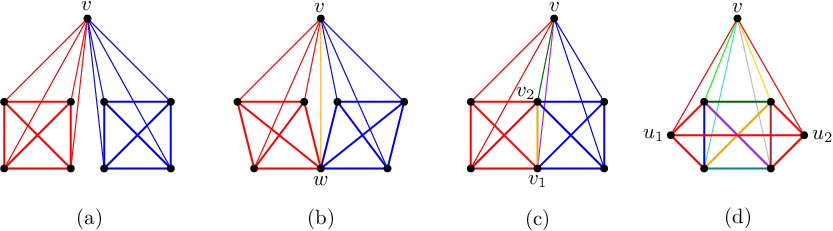}
    \caption{The cones of four double $\cR_2$-circuits.}
    \label{fig:cone+example}
\end{figure}

\begin{example}\label{ex:cones}
Let us illustrate Lemma~\ref{lem:coning-k-circuits-weak} and Theorem~\ref{thm:coning-k-circuits-strong} by coning the four double $\cR_2$-circuits from Example~\ref{ex:dc}.
The resulting graphs are displayed in Figure \ref{fig:cone+example} where $v$ is the added cone vertex.
In Lemma~\ref{lem:coning-k-circuits-weak},
we showed that the cone is a double $\cR_3$-circuit.
In (a), the cone contains exactly two $\cR_3$-circuits, the two copies of $K_5$. So the principal partition of the cone $\{A_1',A_2'\}$ is trivially obtained from $\{A_1,A_2\}$.
In (b), the cone contains two copies of $K_5$ and also a third $\cR_3$-circuit, the double banana $B_{3,2}$, obtained by deleting the edge $wv$.
Hence the principal partition of the cone $\{A_1',A_2',A_3'\}$ extends the partition $\{A_1,A_2\}$ by including a 3rd singleton set $A_3'=\{wv\}$. 
In (c), we obtain the principal partition $\{A_1',A_2',\dots,A_5'\}$ where the two new parts are $A_4'=\{v_1v\}$ and $A_5'=\{v_2v\}$. Thus the edges from the cone vertex $v$ to the two technicolour vertices of the original graph, $v_1$ and $v_2$, belong to different parts of the principal partition of the cone, even though the hypotheses of Theorem~\ref{thm:coning-k-circuits-strong} (a) or (b) are not satisfied. 
Finally in (d), we obtain the principal partition $\{A_1',A_2',\dots, A_{11}'\}$ of the cone from $\{A_1,A_2,\dots, A_{7}\}$ by adding $u_1v,u_2v$ to the part $A_1$ which contains the edges incident to $u_1,u_2$ to form $A_1'$, putting $A_i'=A_i$ for $2\leq i\le 7$ and using the four edges from $v$ to the $K_4$ to form singleton parts $A_8',A_9',A_{10}',A_{11}'$.
\added{Note that the fact $A_8',A_9',A_{10}',A_{11}'$ are all singletons follows from Theorem \ref{thm:coning-k-circuits-strong}(a).}
\end{example}

\begin{remark} \label{rem:coning-counterexample}
Theorem~\ref{thm:coning-k-circuits-strong} may become false if we remove hypotheses (a) and (b).
Consider the double circuit $K_{6,7}$ in $\cR_4$ with bipartition $(X,Y)$ where $|X|=6$ and $|Y|=7$. Recall from Theorem \ref{thm:K_67+dcp} that its principal partition $\{A_1,A_2,\ldots,A_7\}$ is obtained by taking the copies of $K_{1,6}$ centred on the vertices $y_i\in Y$. Thus each vertex in $Y$ is monochromatic and each vertex in $X$ is technicolour. The 
principal partition $\{A_1',A_2',\ldots,A_8'\}$ of its cone $K_{6,7}*v$
 is obtained by putting $A_i'=A_i+vy_i$ for $1\leq i\leq 7$ and taking $A_8'$ to be the set of all edges from $v$ to $X$.
 In particular, observe that $(G*v) \setminus A_8'$ is isomorphic to the $\cR_5$-circuit $K_{7,7}$.
\end{remark}

Remark~\ref{rem:coning-counterexample} demonstrates that Theorem~\ref{thm:coning-k-circuits-strong} may not hold if we remove hypotheses (a) and (b) when $d\geq 4$.
We do not know whether these hypotheses are required when $d=2, 3$ but
the following lemma shows that they are not  required when $d=1$.

\begin{lemma}
\label{lem:coning-k-circuits-strong}
Let $G$ be a $k$-fold \added{$\cR_1$}-circuit, $\cA$ be the principal partition of $G$ and $x,y$ be two technicolour vertices of $G$. Then $vx,vy$ belong to different parts of the principal partition of \added{the $k$-fold $\cR_2$-circuit} $G*v$. %whenever $d=1$.
\end{lemma}
\begin{proof}
Suppose, for a contradiction, that $vx,vy$ belong to the same  part of the principal partition of $G*v$.
Then \added{by Proposition \ref{prop:principalpartfork}(b),}
\begin{equation}\label{eq:same1}
    \mbox{$vy$ is an \added{$\cR_{2}$}-bridge of $G*v-vx$.}
     \end{equation}
In addition, Theorem \ref{thm:coning-k-circuits-strong} gives
%\begin{equation}\label{eq:same2}
    %\mbox{
    $N_G(x)-y=N_G(y)-x=:N$ and $xw,yw$ belong to the same part of $\cA$ for all $w\in N$.
    %}
%     \end{equation}
Since $x$ and $y$ are technicolour vertices of $G$, we also have \added{by Lemma \ref{lem:technicolour+degree} that $|N_G(x)|\geq 3$.}
% \begin{equation}\label{eq:same3}
%     \mbox{$|N_G(x)|\geq d+2$.}
%      \end{equation}
%      Suppose $d=1$. 
    If $xy\not\in E(G)$ then %$|N|\geq 3$ by \eqref{eq:same3} and 
    $G*v$ contains a complete bipartite graph $H$ with bipartition $(\{x,y,v\},\{w_1,w_2,w_3\})$ for any $w_1,w_2,w_3\in N$.
     Then $H+vy$ is an $\cR_2$-circuit in $G*v-vx$, contradicting  \eqref{eq:same1}. Similarly, if $xy\in E(G)$ then $|N|\geq 2$, %by \eqref{eq:same3} 
     and $G*v$ contains a complete bipartite graph $H$ with bipartition $(\{x,y,v\},\{w_1,w_2\})$ for any $w_1,w_2\in N$.
     Then $H+xy+vy$ is an $\cR_2$-circuit in $G*v-vx$, again contradicting  \eqref{eq:same1} and completing the proof of the lemma.
     \end{proof}

We also have the following partial result for $d=2$.

    \begin{lemma}
\label{lem:coning-k-circuits-strong1}
Let $G$ be a double circuit in $\mathcal{R}_2$, $\cA$ be the principal partition of $G$ and $x,y,z$ be three technicolour vertices of $G$. Then $vx,vy,vz$ do not belong to the same part of the principal partition of $G*v$.
\end{lemma}

\begin{proof}
Suppose, for a contradiction, that  $vx,vy,vz$ belong to the same  part of the principal partition of $G*v$.
Then
 Theorem \ref{thm:coning-k-circuits-strong} gives
\begin{equation*}%\label{eq:same4}
    \mbox{$N_G(x)-y-z=N_G(y)-x-z=N_G(z)-x-y=:N$}
\end{equation*}
and 
 \begin{equation}\label{eq:same5}
 \mbox{
 $xw,yw,zw$ belong to the same part of $\cA$ for all $w\in N$.
 }
 \end{equation}
Since $x$ and $y$ are technicolour vertices of $G$, we also have \added{by Lemma \ref{lem:technicolour+degree} that }
\begin{equation}\label{eq:same6}
    \mbox{$|N_G(x)|\geq 4$.}
     \end{equation}

Suppose  $xy\in E(G)$. Then $yz,xz\in E(G)$ by \added{Theorem \ref{thm:coning-k-circuits-strong}}, %\eqref{eq:same4}, 
$|N|\geq 2$ by \eqref{eq:same6} and, since $G$ is a double $\cR_2$-circuit, $G$ is the graph obtained from the triangle on $x,y,z$ by adding two vertices $w_1,w_2$ adjacent to each vertex of this triangle. It is easy to check that  $G*v-\{vx,vy,vz\}$ is $\cR_3$-independent, contradicting the hypothesis that $vx,vy,vz$ belong to the same  part of the principal partition of $G*v$.

Hence  $xy\not\in E(G)$. Then $yz,xz\not\in E(G)$ by \added{Theorem \ref{thm:coning-k-circuits-strong}}, %\eqref{eq:same4}  
and \eqref{eq:same6} now implies that \added{$|N|\geq 4$ and} $G$ contains a subgraph $H\cong K_{3,4}$ with bipartition $(\{x,y,z\},\{w_1,w_2,w_3,w_4\})$ for any $w_1,w_2,w_3,w_4\in N$. Since $G$ is a double $\cR_2$-circuit and $K_{3,4}$ is an $\cR_2$-circuit, $E(G)-E(H)$ is a part of $\cA$. 

Suppose $|N|\geq 5$.
\added{The same argument as the previous paragraph along with} the hypothesis that $G$ is a double $\cR_2$-circuit tells us that $|N|=5$ and $G\cong K_{3,5}$. %\tony{Maybe I miss a subtlely but it's easy I think - We get a $w_5$ as in the previous paragraph, giving a $K_{3,5}$ right? Then it's as stated-  $K_{3,5}$ is a double circuit in R2  and we assume G is a double circuit so it can't contain it as a proper subgraph (also we cant have a $w_6$ as it'd give a triple circuit $K_{3,6}$)}. 
This implies $|\cA|=5$ and each part of $\cA$ induces a copy of $K_{1,3}$ centred on a vertex of $N$.
Theorem \ref{thm:coning-k-circuits-strong} now tells us that $\{vx,vy,vz\}$ is a part of the principal partition of $G*v$.
This is impossible since $G*v$ is a double $\cR_3$-circuit and $G*v-\{vx,vy,vz\}\cong K_{4,5}$ is $\cR_3$-independent.

%Hence $|N|=4$ and  $\cA=\{A_1,A_2,A_3,A_4,A_5\}$ where $A_i$ induces a copy of $K_{1,3}$ centred on $w_i$ for all $1\leq i\leq 4$. Since $G-A_i$ is an $\cR_2$-circuit and $|A_i|=3$ for all $1\leq i\leq 4$, each $w_i$ is monochromatic in $G$. This implies that no edge of $A_5$ is incident to $w_i$ for all $1\leq i\leq 4$. Hence $G$ is disconnected so $G$ is a trivial double circuit. This contradicts the hypothesis that $G$ has technicolour vertices.

Hence $|N|=4$. 
\added{Note that $xw,yw,zw$ belong to the same part of $\cA$ for all $w\in N$ by Equation (\ref{eq:same5}).
Since $x,y,z$ are technicolour these parts are all distinct
and since $E(G)-E(H)$ is a part of $\cA$ this implies that} $\cA = \{A_1,A_2,A_3,A_4,A_5\}$ where $A_i$ induces a copy of $K_{1,3}$ centred on $w_i$ for all $1\leq i\leq 4$.
\added{Since $G$ is a double circuit in $\mathcal{R}_2$ that contains technicolour vertices, it is nontrivial and hence $|E(G)|=2|V(G)|-1$. If, for any $1\leq i \leq 4$, $w_i$ has a neighbour of $G$ distinct from $x,y,z$ then $G-A_i$ would be a $\cR_2$-circuit on $2|V(G)|-4$ edges. This would contradict the fact that $\cR_2$-circuits are $\cR_2$-rigid. Hence 
%$|A_i|=3$ for all $1\leq i\leq 4$ and 
each $w_i$ is monochromatic in $G$.}
This implies that no edge of $A_5$ is incident to $w_i$ for any $1\leq i\leq 4$. Hence $G$ is disconnected so $G$ is a trivial double circuit. This contradicts the hypothesis that $G$ has technicolour vertices.
\end{proof}

We end this section by determining the values of $k$ for which the edges incident to the cone vertex can all belong to the same part of the principal partition in the cone of a $k$-fold circuit.

\begin{corollary}\label{cor:partsincidenttocone}
    Let $G$ be a $k$-fold $\cR_d$-circuit and $G*v$ be its cone.
    If all edges adjacent to $v$ are in the same part of the principal partition of $G*v$, then $k=1$.
\end{corollary}
\begin{proof}
    Write $\cA$ for the principal partition of $G$ and $\cA'$ for the principal partition of $G*v$.
    Suppose at least one vertex $u$ of $G$ is \added{monochromatic}, i.e., all edges adjacent to $u$ are contained in the same part $A \in \cA$.
    Let $U \subseteq V(G)$ be the set of monochromatic vertices in this colour class.
    Lemma~\ref{lem:coning-k-circuits-weak} states that $A' \in \cA'$ where $A'$ is the union of $A$ with all edges from $v$ to $U$.
    As all edges adjacent to $v$ are in the same part by assumption, this implies that $U = V(G)$ and hence $\cA$ and $\cA'$ have exactly one part.
    It follows from Proposition~\ref{prop:principalpartfork} that $G$ and $G*v$ are $\cR_d$-circuits and $\cR_{d+1}$\added{-circuits} respectively.
    
    Now assume that all vertices of $G$ are technicolour.
    If all edges adjacent to $v$ are in the same part of $\cA'$, Theorem~\ref{thm:coning-k-circuits-strong} implies that every for pair of vertices $u,x \in V(G)$, we have $N_G(u) -x = N_G(x) - u$ and $uw$ and $xw$ are in same part of $\cA$ for every $w \in N_G(u) \cap N_G(x)$.
    Fix two vertices $u,x \in V(G)$ such that there exists some $w \in N_G(u) \cap N_G(x)$, then $uw, xw$ are in the same part of $\cA$.
    As $u,x$ are technicolour, there either exists $w' \in N_G(u) \cap N_G(x)$ with $uw', xw'$ in a different part to $uw, xw$, or $ux$ is in a different part to $uw, xw$.
    In the former case, we have $uw, uw'$ in different parts despite $u \in N_G(w) \cap N_G(w')$.
    In the latter case, we have $uw, ux$ in different parts despite $u \in N_G(w) \cap N_G(x)$.
    Both cases give a contradiction.
\end{proof}

\subsection{Almost coning}

We will use our coning result for double circuits to determine when the graph obtained by deleting at most two edges from the cone of a double $\cR_d$-circuit is $\cR_{d+1}$-rigid.
We first give new necessary conditions for a graph obtained by deleting $t$ edges from a coned graph to be minimally $\cR_{d+1}$-rigid.

\begin{lemma} \label{lem:necc-main}
Let $G = (V,E)$ be a graph on at least $d$ vertices and let $G'$ be obtained from the cone $G*v$ by deleting $t$ edges incident to $v$. 
Suppose $G'$ is minimally $\mathcal{R}_{d+1}$-rigid. Then $G$ is $\mathcal{R}_d$-rigid and $|E| = d|V|-\binom{d+1}{2}+t$.
Moreover, if $S$ is the set of vertices of $G'$ that are not adjacent to $v$, then
$G-S$ is $\mathcal{R}_{d}$-independent and each $s \in S$ belongs to an $\mathcal{R}_{d}$-circuit in $G$.
\end{lemma}

\begin{proof}
Since $G'$ is $\mathcal{R}_{d+1}$-rigid, so is $G*v$.
Thus Lemma \ref{lem:coning+rank} implies that $G$ is $\cR_d$-rigid.
Put $G' = (V \cup \{v\}, E')$. As $G'$ is minimally $\cR_{d+1}$-rigid and $|V| \geq d$, we have $|E'| = (d+1)|V\cup\{v\}|-\binom{d+2}{2}$. Hence,
\begin{align*}
    |E| = |E'|-d_{G'}(v)
        &= (d+1)|V \cup \{v\}|-\left(\binom{d+2}{2}+|V|-t\right) \\
        &= d|V|-\left(\binom{d+2}{2}-(d+1+t)\right) \\
        &= d|V|-\binom{d+1}{2}+t.
\end{align*}
The cone $(G-S)*v$ of $G-S$ is a subgraph of $G'$ and so, as $G'$ is minimally $\cR_{d+1}$-rigid, Lemma \ref{lem:coning+rank} implies that $G-S$ is $\cR_d$-independent. 

\added{We finally show that each $s \in S$ belongs to an $\cR_d$-circuit.
Let $H = \{uv : u \in V - \{s\}\}$ be the set of cone edges other than $sv$.
As $H$ is $\cR_{d+1}$-independent and $G'$ is minimally $\cR_{d+1}$-rigid, we can obtain a minimally $\cR_{d+1}$-rigid graph $H'$ containing $H$ by adding only edges from $G$.
Note that $H'+sv$ is $\cR_{d+1}$-rigid and contains a unique $\cR_{d+1}$-circuit $C'$.
Hence $(C'-v)*v$ is $\cR_{d+1}$-dependent.
Therefore Lemma \ref{lem:coning+circuit} implies $C'-v$ is $\cR_{d}$-dependent and hence $G \cap H'$ contains an $\cR_{d}$-circuit $C$.
Since $H'$ is minimally $\cR_{d+1}$-rigid, $C*v$ is not a subgraph of $H'$ and therefore $s \in V(C)$.}

% Observe that $|S| = t$ and put $S = \{s_1, \dots, s_t\}$.
% In order to show that $s_{i}$ belongs to an $\cR_{d}$-circuit in $G$ for all $1 \leq i \leq t$, we proceed by induction on $t$.
% For the base case, let $t = 1$.
% Note that $G'+vs_{1}$ is $\cR_{d+1}$-rigid and contains a unique $\cR_{d+1}$-circuit $C'$.
% Hence $(C'-v)*v$ is $\cR_{d+1}$-dependent.
% Therefore Lemma \ref{lem:coning+circuit} implies $C'-v$ is $\cR_{d}$-dependent and hence $G$ contains an $\cR_{d}$-circuit $C$.
% Since $G'$ is minimally $\cR_{d+1}$-rigid, $C*v$ is not a subgraph of $G'$ and therefore $s_{1} \in V(C)$. 

% Suppose that there exists $k \geq 1$ such that if $G'$ is minimally $\cR_{d+1}$-rigid and $t \leq k$ then $s_{i}$ belongs to an $\cR_d$-circuit in $G$ for all $1 \leq i \leq t$. Now let $G'$ be minimally $\cR_{d+1}$-rigid and suppose that $t = k+1$. Take $1 \leq i \leq k+1$ and note that $G'+vs_{i}$ is $\cR_{d+1}$-rigid but is not minimally $\cR_{d+1}$-rigid. It follows that there exists a unique subgraph $C$ of $G'+vs_{i}$ that is an $\cR_{d+1}$-circuit. Take $e \in E(C) \cap E(G)$. Then $G'+vs_{i}-e$ is minimally $\cR_{d+1}$-rigid and so, by the induction hypothesis, $s_{j}$ belongs to an $\cR_d$-circuit in $G-e$ for all $1 \leq j \leq k+1$ such that $j \neq i$. Since $i$ was arbitrary, $s_{j}$ belongs to an $\cR_{d}$-circuit in $G$ for all $1 \leq j \leq k+1$. 
\end{proof}

\added{While the conditions in Lemma \ref{lem:necc-main} are necessary for $G'$ to be minimally $\cR_{d+1}$-rigid, they are not sufficient when $t \geq 3$ as the following example demonstrates.}
 Let $H$ be a minimally $\cR_d$-rigid graph on $d+3$ vertices.
Since 
$$ \frac{(d+3)(d+2)}{2}-\left(d(d+3)-{d+1\choose 2}\right)=3, $$
one can obtain the complete graph $K_{d+3}$ from $H$ by adding $3$ edges. Since $K_{d+3}$ is an $\cR_{d+1}$-circuit, the converse to Lemma \ref{lem:necc-main} fails when $t\geq 3$. 
\added{Under some additional assumptions, we will prove the conditions in  Lemma \ref{lem:necc-main} are also sufficient when $t\leq 2$.}

\begin{theorem} \label{thm:maint=2}
Let $G=(V,E)$ be a graph on at least $d$ vertices and let $G'$ be obtained from the cone $G*v$ by deleting $t \leq 2$ edges incident to $v$.
Let $S$ be the set of vertices of $G$ that are not adjacent to $v$.
\added{If $S=\{s_1,s_2\}$ then assume that $N_G(s_1)-s_2\neq N_G(s_2)-s_1$.}
Then $G'$ is minimally $\mathcal{R}_{d+1}$-rigid if and only if
$G$ is $\mathcal{R}_d$-rigid,  $|E| = d|V|-\binom{d+1}{2}+t$,
$G-S$ is $\mathcal{R}_{d}$-independent, and each $s \in S$ belongs to an $\mathcal{R}_{d}$-circuit in $G$. 
\end{theorem}

\begin{proof}
The necessity is Lemma \ref{lem:necc-main}. 
The case when $t=0$ is Lemma \ref{lem:coning+rank}.

Suppose $t=1$ and put $S=\{s_1\}$. Then $G$ is $\cR_d$-rigid and $|E| = d|V|-{d+1\choose 2}+1$, so $G$ contains a unique $\cR_d$-circuit $C$.
Moreover, as $G-\{s_1\}$ is $\cR_{d}$-independent we must have $s_1 \in V(C)$.
Lemma \ref{lem:coning+circuit} implies that $G*v$ contains a unique $\cR_{d+1}$-circuit $C*v$ and $s_1 \in V(C*v)$.
Hence removing $vs_1$ gives the minimally $\cR_{d+1}$-rigid graph $G'$.

It remains to consider the case when $t=2$. Put $S=\{s_1,s_2\}$. Then $G$ is $\mathcal{R}_d$-rigid and $|E| = d|V|-\binom{d+1}{2}+2$, so $G$ contains a unique double circuit $D$ in $\mathcal{R}_d$, \added{obtained by removing the $\cR_d$-bridges from $G$.
As each $s \in S$ is contained in an $\cR_d$-circuit, it follows that $s_1,s_2\in V(D)$.}
%Furthermore, $s_1,s_2\in V(D)$ as $G-S$ is $\cR_d$-independent.
Lemma \ref{lem:coning-k-circuits-weak} now implies that $G*v$ contains a unique double circuit $D*v$ in $\mathcal{R}_{d+1}$ and $s_1,s_2\in V(D*v)$. 
It remains to show that $vs_1,vs_2$ are in different parts of the principal partition $\cA'$ of $D*v$. 
Suppose $vs_1,vs_2$ are in the same class $A'$ of $\cA'$. 
Since $N_G(s_1)-s_2\neq N_G(s_2)-s_1$,
Theorem \ref{thm:coning-k-circuits-strong} now implies that $s_1,s_2$ are monochrome and all edges incident to $s_1,s_2$ are in the same class of the principal partition of $D$.
This contradicts the hypothesis that $G-S$ is $\cR_d$-independent.
\end{proof}

\begin{remark}
Theorem~\ref{thm:maint=2} may no longer hold if we drop the hypothesis $N_G(s_1)-s_2\neq N_G(s_2)-s_1$.
Recall the flexible double circuit $K_{6,7}$ in $\cR_4$ with bipartition $(X,Y)$ where $|X| = 6$ and $|Y|=7$.
It has two degrees of freedom, so \added{for $y_1, \dots y_4 \in Y$ we have} $G = K_{6,7} + \added{\{y_1y_2, y_3y_4\}}$ is $\cR_4$-rigid where the two new edges are bridges.
Setting $S = \{x_1,x_2\} \added{\subset X}$, we have \added{that} $G - S$ is $\cR_4$-independent.
However, the graph $G'$ obtained from $G*v$ by removing $vx_1$ and $vx_2$ is not $\cR_5$-independent.
This follows as $K_{6,7}*v$ is a double $\cR_5$-circuit, and Remark~\ref{rem:coning-counterexample} computes that $vx_1$ and $vx_2$ are contained in the same part of the principal partition, hence $G'$ contains a $\cR_5$-circuit.
\end{remark}

\begin{example}\label{ex:coning-3-fold}
    One cannot extend Theorem~\ref{thm:maint=2} to $t=3$ using our methods alone, as the following example demonstrates.
    Consider the $3$-fold circuit $G$ in $\cR_2$ and its cone $G*v$ depicted in Figure~\ref{fig:conepartitionexample}.
    The principal partition of $G$ is $\{A_1, A_2, \dots, A_{10}\}$ where $A_1$ is the 5 edges incident to $x$ and $y$, and $A_2,\dots, A_{10}$ are singleton classes.
    Theorem~\ref{thm:coning-k-circuits-strong} implies $G*v$ is a $3$-fold $\cR_3$-circuit, and that its principal partition is formed by taking $A_1$ and appending $xv,yv$, then the remaining edges incident to $v$ are all singleton classes.

    Letting $S = \{a,b,c\}$, we see that $G$ is $\cR_2$-rigid, $G-S$ is \added{$\cR_2$}-independent and each $s \in S$ belongs to a $\cR_2$-circuit, satisfying the conditions of Theorem~\ref{thm:maint=2}.
    However $G':= G*v - \{av, bv, cv\}$ is the double banana and hence not $\cR_{3}$-rigid.
    The issue is that while removing any two of $\{av,bv,cv\}$ does not decrease the rank, removing all three does decrease the rank.
    In the language of $k$-fold circuits, each is a singleton class in the principal partition of the $3$-fold circuit $G*v$.
    However, in the double circuit $G' - {av}$, the edges $bv, cv$ are now in the same part of its principal partition.
\end{example}

Example~\ref{ex:coning-3-fold} highlights that extending Theorem~\ref{thm:maint=2} to $t > 2$ requires understanding of all $k$-fold circuits contained in a $t$-fold circuit for all $k \leq t$, rather than just the principal partition.
By~\cite[Lemma 3.5]{JNS}, these $k$-fold circuits form a lattice isomorphism to the lattice of flats of a rank $t$ matroid.
When $t=2$, this lattice is very simple: the proper flats are exactly the parts of the principal partition.
Even for $t=3$, this lattice can be highly complex.
For example, for a $3$-fold circuit whose principal partition has $\ell = 12$ parts, there are over 28 million non-isomorphic lattice structures one could have, and the number is unknown for $\ell > 12$~\cite{BBJKL}.
As such, extending Theorem~\ref{thm:maint=2} to $t > 2$ seems to require techniques beyond the methods of this article.

\begin{figure}[ht]
\begin{center}
\begin{tikzpicture}[scale=0.8]

\draw[black,thick]
(0,0) -- (2.5,3) -- (2.5,1) -- (1,2) -- (2.5,3) -- (0,4) -- (-2.5,3) -- (-2.5,1) -- (-1,2) -- (-2.5,3) -- (0,0) -- (2.5,1) -- (0,4) -- (1,2) -- (0,0) -- (-1,2) -- (0,4) -- (-2.5,1) -- (0,0);

\draw[black,thick]
(1,2) -- (-2.5,3) -- (2.5,3);

\draw[black,thick]
(-2.5,3) -- (2.5,1);

\filldraw (0,0) circle (3pt)node[anchor=north]{$z$};
\filldraw (0,4) circle (3pt)node[anchor=south]{$w$};

\filldraw (2.5,3) circle (3pt)node[anchor=west]{$a$};
\filldraw (2.5,1) circle (3pt)node[anchor=west]{$b$};
\filldraw (1,2) circle (3pt)node[anchor=east]{$c$};

\filldraw (-2.5,3) circle (3pt)node[anchor=east]{$v$};
\filldraw (-2.5,1) circle (3pt)node[anchor=east]{$x$};
\filldraw (-1,2) circle (3pt)node[anchor=west]{$y$};

%%%%%

\draw[black,thick]
(7,0) -- (9.5,3) -- (9.5,1) -- (8,2) -- (9.5,3) -- (7,4);

\draw[black,thick]
(4.5,1) -- (6,2);

\draw[black,thick]
(7,0) -- (9.5,1) -- (7,4) -- (8,2) -- (7,0) -- (6,2) -- (7,4) -- (4.5,1) -- (7,0);

\filldraw (7,0) circle (3pt)node[anchor=north]{$z$};
\filldraw (7,4) circle (3pt)node[anchor=south]{$w$};

\filldraw (9.5,3) circle (3pt)node[anchor=west]{$a$};
\filldraw (9.5,1) circle (3pt)node[anchor=west]{$b$};
\filldraw (8,2) circle (3pt)node[anchor=east]{$c$};

\filldraw (4.5,1) circle (3pt)node[anchor=east]{$x$};
\filldraw (6,2) circle (3pt)node[anchor=west]{$y$};

\end{tikzpicture}
\end{center}
\caption{The left graph $G*v$ is the cone of the right graph $G$, a 3-fold $\cR_2$-circuit.}
\label{fig:conepartitionexample}
\end{figure}
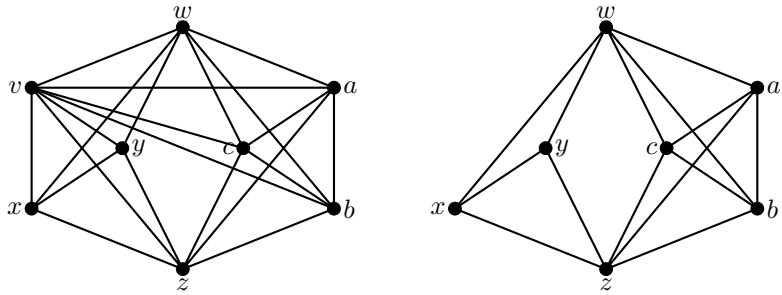

\subsection{Adding edges to $\cR_d$-independent graphs}

As a second application of Theorem \ref{thm:coning-k-circuits-strong} we have the following.

\begin{corollary}\label{cor:addingedges}
Let $G$ be obtained from an $\cR_d$-independent graph by adding at most 2 edges. Then $G$ is $\cR_{d+1}$-independent.    
\end{corollary}

\begin{proof}
Let $G*v$ denote the cone of $G$ with new vertex $v$.
Since $G$ is a subgraph of $G*v$, the cases when no edges are added and 1 edge is added follow immediately from Lemmas \ref{lem:coning+rank} and \ref{lem:coning+circuit} respectively.
Suppose 2 edges are added: if either \added{is} an $\cR_d$-bridge then this reduces to the case where 1 edge is added.
Hence we may assume that neither edge is an $\cR_d$-bridge. Then $G$ contains a unique double $\cR_d$-circuit.
By Lemma \ref{lem:coning+rank} \added{$\mathcal{R}_d$-bridges are also $\mathcal{R}_{d+1}$-bridges, hence} it suffices to consider the case when $G$ is a double $\cR_d$-circuit. Lemma \ref{lem:coning-k-circuits-weak} implies that $G*v$ is a double $\cR_{d+1}$-circuit.
Furthermore, Corollary \ref{cor:partsincidenttocone} implies that there exists two edges incident to $v$ in different parts of the principal partition of $G*v$.
Removing these edges gives a $\cR_{d+1}$-independent graph containing $G$.
\end{proof}

Note that if $G$ is obtained by adding 3 edges to an $\cR_d$-independent graph then $G$ may not be $\cR_{d+1}$-independent. For example we can form a subgraph isomorphic to $K_{d+3}$ by adding 3 edges to an $\cR_d$-independent graph.

\subsection{$X$-replacement}

A graph $G'$ is said to be obtained from another graph $G$ by a \emph{($d$-dimensional) $X$-replacement} if there exists non-adjacent edges $uv$ and $xy$ in $G$ such that $G'$ is $G - \{uv, xy\}$ plus an additional vertex $w$ of degree $d+2$ adjacent to $u,v,x,y$.
%    \item \arxiv{a \emph{($d$-dimensional) $V$-replacement} if there exists adjacent edges $xy$ and $yz$ in $G$ such that $G'$ is $G - \{xy, yz\}$ plus an additional vertex $w$ of degree $d+2$ adjacent to $x,y,z$.}
See Figure \ref{fig:xvpic} for an illustration of this operation in the case when $d=3$.

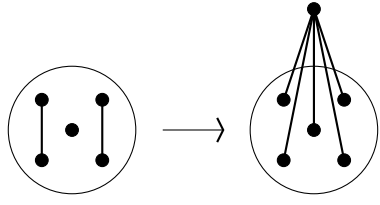
\begin{figure}
\begin{center}
\begin{tikzpicture}[scale=0.8]

\draw (0,0) circle (30pt);
\draw (4,0) circle (30pt);

\filldraw (4,2) circle (3pt);

\filldraw (.5,-.5) circle (3pt);
\filldraw (.5,.5) circle (3pt);
\filldraw (4.5,-.5) circle (3pt);
\filldraw (4.5,.5) circle (3pt);

\filldraw (-.5,-.5) circle (3pt);
\filldraw (-.5,.5) circle (3pt);
\filldraw (3.5,-.5) circle (3pt);
\filldraw (3.5,.5) circle (3pt);

\filldraw (0,0) circle (3pt);
\filldraw (4,0) circle (3pt);

\draw[black,thick]
(.5,.5) -- (.5,-.5);

\draw[black,thick]
(-.5,.5) -- (-.5,-.5);

\draw[black,thick]
(4.5,-.5) -- (4,2) -- (4.5,.5);

\draw[black,thick]
(3.5,-.5) -- (4,2) -- (4,0);

\draw[black,thick]
(3.5,.5) -- (4,2);

\draw[black]
(1.5,0) -- (2.5,0);

\draw[black,thick]
(2.4,-.2) -- (2.5,0) -- (2.4,.2);

\end{tikzpicture}
\caption{Illustration of 3-dimensional $X$-replacement.}
\label{fig:xvpic}
\end{center}
\end{figure}

\added{Tay and Whiteley \cite{TW} conjectured that the $X$-replacement operation preserves minimal $\cR_3$-rigidity.}

\begin{conjecture}\label{con:xrep}
Suppose $G$ is minimally $\cR_3$-rigid, and $G'$ is obtained from $G$ by a $3$-dimensional $X$-replacement.
Then $G'$ is minimally $\cR_3$-rigid.
\end{conjecture}

%It is not hard to see that $V$-replacement does not always preserve minimal $\cR_3$-rigidity. \arxiv{However, the double-V-conjecture \cite{TW} gives a stronger statement which may preserve minimal $\cR_3$-rigidity.}

%\begin{conjecture}\label{con:vrep}
%Suppose $H\cup \{e,f\}$ and $H\cup \{e',f'\}$ are minimally $\cR_3$-rigid for pairs of adjacent edges $e,f$ and $e',f'$ such that the common endvertex of $e,f$ is distinct from the common endvertex of $e',f'$. Let $G$ be the graph obtained from $H$ by adding a new vertex $v$ of degree 5 such that $N(v)$ contains the endvertices of $e,f,e',f'$. Then the double V-conjecture claims that $G$ is minimally $\cR_3$-rigid.
%\end{conjecture}

%Conjectures \ref{con:xrep} and \ref{con:vrep} are of key importance in the study of $\mathcal{R}_3$-rigidity since they would imply that all minimally $\cR_3$-rigid graphs can be generated recursively by these local operations alongside 0- and 1-extension; \added{see \cite[Section 4]{CJJT} for more details.}

\added{The fact that X-replacement preserves independence in the $C_2^1$-cofactor matroid is a key component in the characterisation of this matroid in \cite{CJT0,CJT1}, so a resolution of this conjecture would be a big step towards deciding whether $C_2^1=\cR_3$.} 
We show that the $X$-replacement conjecture can be stated in terms of double circuits. 
Suppose $G$ is minimally $\cR_3$-rigid and $e=uv$ and $f=xy$ are two non-adjacent edges of $G$.
Let $H$ be obtained from $G$ by adding a vertex $w$ of degree 5 joined to $u,v,x,y$ and one other vertex $z$.
Then $H$ contains a unique double circuit $D$ in $\cR_3$ and $w$ is a technicolour vertex of $D$ since $D-w\subseteq G-w$ is  $\cR_3$-independent.
In addition $H-e-f$ is minimally $\cR_3$-rigid if and only if $e,f$ belong to different sets in the principal partition of $D$.
Hence the $X$-replacement conjecture can be restated as follows.

\begin{conjecture}\label{con:doubleX}
 Let $D$ be a double $\cR_3$-circuit and $w$ be a technicolour vertex of degree 5 in $D$.
 Then every pair of non-adjacent edges between 4 vertices in the neighbourhood of $w$ are in different parts of the principal partition of $D$.
\end{conjecture}

The following result gives partial information.

\begin{lemma}\label{lem:xrepcases}
 Let $D=(V,E)$ be a double $\cR_d$-circuit and $w$ be a technicolour vertex of degree $d+2$ in $D$ with neighbour set $\{u_1,u_2,\dots,u_{d+2}\}$.
 For each $e\in E$, let $A_e$ be the part of the principal partition of $D$ which contains $e$. Then $A_{wu_i}\neq A_{wu_j}$ for all $1\leq i<j\leq d+2$. In addition:
 \begin{enumerate}
 \item[(a)] if $u_iu_j\in E$ for some $1\leq i< j\leq d+2$, then $A_{u_iu_j}\neq A_{wu_k}$ for all $k\neq i,j$;
\item[(b)] if $u_1u_2,u_3u_4\in E$ and $A_{u_1u_2}= A_{u_3u_4}$, then $A_{u_1u_2}\neq  A_{wu_k}$ for all $1\leq k\leq d+2$.
 \end{enumerate}
\end{lemma}

\arxiv{\begin{proof}
We first show that $A_{wu_i}\neq A_{wu_j}$ for all $1\leq i<j\leq d+2$. Suppose for a contradiction that $A_{wu_i}= A_{wu_j}=A_1$. Since $w$ is technicolour, $w$ has degree at least one and at most $d$ in $D-A_1$. As $\cR_d$-circuits have minimum degree at least $d+1$, this contradicts the fact that $D-A_1$ is an $\cR_d$-circuit. Hence  $A_{wu_i}\neq A_{wu_j}$ for all $1\leq i<j\leq d+2$.   Put $A_{wu_i}=A_i$ for all $1\leq i\leq d+2$. 

To prove (a) we assume for a contradiction that $u_1u_2\in E$ and  $A_{u_1u_2}= A_3$. Then $D-u_1u_2-wu_3$
contains an $\cR_d$-circuit. 
This contradicts the fact that $D-w$ is $\cR_d$-independent, as $w$ is a 
technicolour vertex of $D$, and $D-u_1u_2-wu_3$ can be obtained from $D-w$ by a 1-extension.

Part (b) of the lemma follows immediately from part (a).
\end{proof}

Note that Conjectures \ref{con:xrep} and \ref{con:doubleX} do not hold for double $\cR_4$-circuits. To see this, consider the complete bipartite graph $K_{6,6}$ with partition $A\cup B$ where $A=\{v_1,v_2,\dots,v_6\}$ and $B=\{u_1,u_2,\dots,u_6\}$.
Let $H$ be obtained from this copy of $K_{6,6}$ by adding the edges $v_2v_3,v_4v_5, u_3u_4, u_5u_6$ and deleting the edges $v_1u_1$ and $v_1u_2$. It is easy to find a sequence of 0- and 1-extensions applied to $K_5$ that results in $H$. Hence Lemma \ref{lem:01ext} implies that $H$ is minimally $\cR_4$-rigid. It follows that the graph $D$ obtained from $H$ by adding $v_1u_1$ and $v_1u_2$ is a double $\cR_4$-circuit, and that $v_1$ is technicolour. 
Since $D-\{v_2v_3, v_4v_5, u_3u_4, u_5u_6\}=K_{6,6}$ is an  $\cR_4$-circuit, $\{v_2v_3, v_4v_5, u_3u_4, u_5u_6\}$ is a set in the principal partition of $D$.
Hence $v_1$ is a technicolour vertex of degree $6$ in $D$ and the pair of non-adjacent edges $\{u_3u_4,u_5u_6\}$ in the neighbourhood of $v_1$ are in the same part of the principal partition.}

%\textcolor{red}{Ben, why did you put the following back? It does'nt seem to be what the referee requested}

%\arxiv{We conclude this section by deducing a special case of both conjectures. In particular, we use Corollary \ref{cor:addingedges} to prove a weaker statement than Conjecture \ref{con:xrep} that holds in arbitrary dimension.

%\begin{lemma}
%Let $G$ be $\mathcal{R}_{d}$-independent and let $H$ be formed from $G$ by deleting an edge $uv$ and then adding a new vertex $w$ of degree $d+3$ adjacent to $u$ and $v$.
%Then $H$ is $\cR_{d+1}$-independent.   
%\end{lemma}

%\begin{proof}
%Suppose that $G'$ is the spanning subgraph of $H$ obtained from $G$ by a $d$-dimensional 1-extension on the edge $uv$ adding $w$. Since $G$ is $\cR_d$-independent, $G'$ is $\cR_d$-independent by Lemma \ref{lem:01ext}. Since $H$ can be obtained from $G'$ by adding two edges incident to $w$, Corollary \ref{cor:addingedges} implies that $H$ is $\cR_{d+1}$-independent.   
%\end{proof}

%Note that the lemma immediately implies that both $(d+1)$-dimensional $X$- and $V$-replacement preserve $\cR_{d+1}$-independence when applied to $\cR_d$-independent graphs.

%\begin{corollary}\label{cor:xrep}
%Let $G$ be $\mathcal{R}_d$-independent and let $G'$ be obtained from $G$ by applying a $(d+1)$-dimensional $X$- or $V$-replacement. Then $G'$ is $\mathcal{R}_{d+1}$-independent.    
%\end{corollary}}

\subsection{Small $k$-fold $\cR_d$-circuits}

The $k$-fold $\cR_d$-circuits with the fewest vertices arise from complete graphs.
When $n>d+1$, $K_n$ is a $k$-fold $\cR_d$-circuit where $2k=n^2-(2d+1)n+d(d+1)$. 
If $k$ does not satisfy this equation, then the smallest $k$-fold $\cR_d$-circuits are obtained from the appropriate complete graph by deleting a small number of edges.
All such $k$-fold $\cR_d$-circuits are $\cR_d$-rigid.
In the following, we will determine the smallest flexible double $\cR_d$-circuit.
This extends the results of \cite{GGJN} which determined the smallest flexible $\cR_d$-circuits.

Recall the graph $B_{d,d-1}$ defined in Example \ref{ex:db}.
From \cite[Lemma 11]{GGJN} and its preceding discussion, it follows that $B_{d,d-1}$ is a flexible $\cR_d$-circuit with $r_d(B_{d,d-1})=d(d+5)-{d+1\choose 2}-1$.
Define $\overline B_{d,d-1}$ to be obtained from $B_{d,d-1}$ by adding back the edge $e$.
It follows that $\cl_d (B_{d,d-1})=\overline B_{d,d-1}$, and so $\overline B_{d,d-1}$ is a flexible double $\cR_d$-circuit.

\begin{lemma}\label{lem:small+flex+double+circuit}
For $k\geq 2$, suppose $G=(V,E)$ is a flexible $k$-fold $\cR_d$-circuit on at most $d+5$ vertices.
Then $k=2$ and $G=\overline B_{d,d-1}$.    
\end{lemma}

\arxiv{\begin{proof}
We first consider the case when $k=2$, i.e. $G$ is a double $\cR_d$-circuit.

 Suppose that every $\cR_d$-circuit in $G$ is $\cR_d$-rigid. Then, by Lemma \ref{lem:dcondouble} and the hypothesis that $G$ is flexible, $G$ is not $d$-connected. 
 Let $C_1,C_2$ be $\cR_d$-circuits in $G$. Then $G=C_1\cup C_2$.
 Since every $\cR_d$-circuit has at least $d+2$ vertices, we have $|V|\geq d+2+d+2-(d-1)=d+5$. Since $|V|\leq d+5$, equality must hold in both inequalities implying that $C_i\cong K_{d+2}$ for $i=1,2$. Thus $G=\overline B_{d,d-1}$. Then $G$ contains the flexible $\cR_d$-circuit $B_{d,d-1}$ contradicting the assumption that every $\cR_d$-circuit in $G$ is $\cR_d$-rigid.

Hence $G$ contains a flexible $\cR_d$-circuit $C$. By \cite[Theorem 1]{GGJN} and the hypothesis that $|V|\leq d+5$ we have that $C=B_{d,d-1}$. Since both $\cl_d (B_{d,d-1})=\overline B_{d,d-1}$ and $G$ are flexible double $\cR_d$-circuits on the same vertex set, we now have that $G=\overline B_{d,d-1}$.

We next consider the case when $k>2$. Let $A$ be a class in the principal partition of $G$. Proposition \ref{prop:principalpartfork} implies that $G-A$ is a $(k-1)$-fold $\cR_d$-circuit, hence has its own principal partition. We iterate this, deleting an arbitrary class in each successive principal partition until we obtain a double circuit $H$ in $\cR_d$.
 Suppose first that $H$ is a spanning subgraph of $G$.
 If $H=\overline B_{d,d-1}$, then as $\cl_d (\overline B_{d,d-1})=\overline B_{d,d-1}$, we have $r_d(G) >r_d(\overline B_{d,d-1})$. As $r_d(\overline B_{d,d-1})=d(d+5)-{d+1\choose 2}-1$, $G$ is $\cR_d$-rigid, contradicting the hypotheses. 

Hence we may suppose that $d+2\leq |V(H)|\leq  d+4$. The case when $k=2$ implies that $H$ is $\cR_d$-rigid.
Recall that $G$ has minimum degree $d+1$ by Lemma \ref{lem:maxk}.
If $|V(H)|\in \{d+3,d+4\}$ then we can obtain an $\cR_d$-rigid spanning subgraph $G' \subseteq G$ from $H$ via 0-extensions, contradicting that $G$ is $\cR_d$-flexible.
Thus $|V(H)|=d+2$ and $H\cong K_{d+2}$. 
If $|V|\leq d+4$, then the same argument shows $G$ is $\cR_d$-rigid. Hence $|V|=d+5$. 

Let $S$ denote the set of 3 vertices in $G$ that are not in $H$.
If any vertex in $S$ has at least $d$ neighbours in $H$, we can again obtain an $\cR_d$-rigid spanning subgraph of $G$ from $H$ via a series of 0-extensions.
Hence $G[S] \cong K_3$ and each vertex in $S$ has exactly $d-1$ neighbours in $H$.
If the set of all neighbours of $S$ in $H$ has size $d-1$, then since $k>2$, $G$ contains $\overline B_{d,d-1}$ as a proper subgraph and hence Lemma \ref{lem:intbridge}(\ref{it:intbridge:rank}) implies that $G$ is $\cR_d$-rigid.
If the set of all neighbours of $S$ in $H$ has size at least $d$, we may use \cite[Lemma 6]{GGJN} to deduce that $G$ is $\cR_d$-rigid. 
Either case contradicts the assumption that $G$ is $\cR_d$-flexible, completing the proof. 
\end{proof}}

% \noindent 4. Lemma 4.14 gives a lower bound on the minimal number of vertices a flexible $k$-fold circuit can have.
% This bound is tight when $k=2$ and is attained by $\overline{B}_{d,d-1}$, the union of two copies of $K_{d+2}$ that share a common copy of $K_{d-1}$.
% It would be interesting to find a tight bound for $k\geq 3$.\\
Lemma \ref{lem:small+flex+double+circuit} gives a lower bound on the minimal number of vertices a flexible $k$-fold circuit can have, and this bound is tight when $k=2$. %and is attained by $\overline{B}_{d,d-1}$, the union of two copies of $K_{d+2}$ that share a common copy of $K_{d-1}$.
It would be interesting to find a tight bound for $k\geq 3$.

\section{Concluding remarks}
\label{sec:conclusion}

\paragraph{Double circuit property in $\cR_3$} We have established that $\cR_d$ with $d\leq 2$ has the double circuit property and $\cR_d$ with $d\geq 4$ does not.
The problem of determining if $\cR_3$ has the double circuit property is still open, but it may be tractable to determine if this property holds for the closely related $C_2^1$-cofactor matroid. It is conjectured that this matroid is isomorphic to $\cR_3$ and recently the rank function of this matroid has been characterised \cite{CJT0}.

\paragraph{Generalised 2-sums} Given Lemma \ref{lem:k-sum}, it is natural to consider the generalisation of 2-sums to allow more vertices in the intersection.
\added{We say that a graph $G$ is the \emph{$t$-sum} of two graphs $G_1$ and $G_2$ along an edge $e \in E(G_1) \cap E(G_2)$ if $G = (G_1 \cup G_2) - e$ and $G_1 \cap G_2 \cong K_t$ is a complete graph on $t$ vertices.
Garamv{\"o}lgyi \cite[Theorem 5.7]{Gstress} proved that for $2 \leq t \leq d+1$, the $t$-sum of two $\cR_d$-circuits is again an $\cR_d$-circuit.
However, the converse is not true: see \cite[Figure 1]{GGJNc} for a counterexample.

In an earlier version of this paper, we asked whether Garamv{\"o}lgyi's statement could extend to $k$-fold circuits.
Garamv{\"o}lgyi has since provided us with a proof of this fact, which we detail below.}
% If $G$ is an $\cR_d$-circuit and $G=G_1\cup G_2$ where $G_1\cap G_2 \cong K_{t}-e$ for some $3\leq t\leq d+1$ then it need not be true that $G_1$ and $G_2$ are $\cR_d$-circuits (see \cite[Figure 1]{GGJNc} for a counterexample). However Garamv{\"o}lgyi \cite[Theorem 5.5]{Gstress} proved that the converse is true. It would be interesting to determine if his statement extends to $k$-fold $\cR_d$-circuits.

\added{
\begin{lemma}\label{lem:t-sum-kfold}
    Let $2 \leq t \leq d+1$ and $G = (V,E)$ be the $t$-sum of two graphs $G_1 = (V_1, E_1)$ and $G_2 = (V_2,E_2)$ along an edge $e \in E_1 \cap E_2$.
    If $G_1$ is a $k_1$-fold $\cR_d$-circuit and $G_2$ a $k_2$-fold $\cR_d$-circuit, then $G$ is a $k$-fold $\cR_d$-circuit where $k+1 = k_1 + k_2$.
\end{lemma}

\begin{proof}
    Let $G' = G_1 \cup G_2$, clearly this is $\cR_d$-cyclic.
    Applying \cite[Theorem 5.5]{Gstress}, we obtain that $G$ is also $\cR_d$-cyclic.
    To see that $G$ has the correct rank, observe that $G_1 \cap G_2 \cong K_t$ is $\cR_d$-independent and $\cR_d$-rigid.
    As $G_i$ is a $k_i$-fold $\cR_d$-circuit for each $i=1,2$, there exists $F_i \subseteq E_i \setminus (E_1 \cap E_2)$ such that $G_i - F_i$ is $\cR_d$-independent and $|F_i| = k_i$.
    We see from Lemma~\ref{lem:intbridge}\eqref{it:intbridge:indep} that $G' - F_1 - F_2$ is $\cR_d$-independent and hence
    \[
    r(G') \geq r(G' - F_1 - F_2) = |E(G' - F_1 - F_2)| = |E| + 1 - k_1 - k_2 \, .
    \]
    The converse inequality holds by submodularity:
    \[
    r(G') \leq r(G_1) + r(G_2) - r(G_1 \cap G_2) = |E_1| - k_1 + |E_2| - k_2 - |E_1 \cap E_2| = |E| + 1 - k_1 - k_2 \, .
    \]
    Finally, we note $r(G) = r(G')$ as $G'$ is $\cR_d$-cyclic.
\end{proof}
}

\paragraph{Relating $\cR_d$ to $\cR_{d+1}$} Given a graph $G$ on $n$ vertices with a vertex $v$ of degree at least $n-3$, \added{under some additional structural assumptions} Theorem \ref{thm:maint=2} characterises the \added{$\cR_{d+1}$}-rigidity of $G$ in terms of the \added{$\cR_{d}$}-rigidity properties of $G-v$.
It would be interesting to extend the theorem to express the \added{rank of $G$ in $\cR_{d+1}$ in terms of the rank of $G-v$ in $\cR_{d}$.}
%function of \added{$\cR_{d+1}$} in terms of the rank function of \added{$\cR_{d}$}
for such graphs.

\paragraph{Relaxing hypotheses of Theorem \ref{thm:coning-k-circuits-strong}} We saw that Theorem \ref{thm:coning-k-circuits-strong} does not hold without hypotheses (a) and (b) when $d \geq 4$.
However, we showed in Lemma~\ref{lem:coning-k-circuits-strong} that when $d=1$, we can remove hypotheses (a) and (b) from Theorem \ref{thm:coning-k-circuits-strong}.
It would be interesting to see if these can also be dropped for $d=2,3$.
Finally, the only counterexample we know of when $d \geq 4$ is flexible, hence one may be able to also remove the hypotheses for $\cR_d$-rigid graphs.

\section*{Acknowledgements}

\added{We thank two anonymous referees for their careful reading and helpful comments.
We also thank Daniel Garamv{\"o}lgyi for the proof of Lemma \ref{lem:t-sum-kfold}.}
A.\,N.\ and B.\,S.\ were partially supported by EPSRC grant EP/X036723/1.
The research visit that began this project was funded by the Heilbronn Institute for Mathematical Research small grant `Counting realisations of discrete structures'.


\begin{thebibliography}{99}

\bibitem{asi-rot} L. Asimow and B. Roth, The rigidity of graphs, \emph{Transactions of the American Mathematical Society}, 245 (1978) 279--289.

\bibitem{BBJKL} M. Barakat, R. Behrends, C. Jefferson, L. K\"uhne, and M. Leuner, On the Generation of Rank 3 Simple Matroids with an Application to Terao's Freeness Conjecture, \emph{SIAM Journal on Discrete Mathematics}, 35:2 (2021).

\bibitem{BJ} A. Berg and T. Jord\'an, A proof of Connelly’s conjecture on 3-connected circuits of the rigidity matroid, \emph{Journal of Combinatorial Theory: Series B}, 88:1 (2003) 77--97.

\bibitem{Brylawski} T. Brylawski, Constructions, in \emph{Theory of Matroids}, N. White, Ed. Cambridge: Cambridge University Press (1986)  127–-223 

\bibitem{CJJT} J. Cruickshank, B. Jackson, T. Jord\'an and S. Tanigawa, Rigidity of Graphs and Frameworks: A Matroid Theoretic Approach, in Survey's in Combinatorics 2026, A. Gagarin ed., CUP, 2026,  189--230. 
%arXiv:2508.11636.

\bibitem{CJT0} K. Clinch, B. Jackson and S. Tanigawa, Abstract 3-Rigidity and Bivariate $C_2^1$-Splines I: Whiteley’s Maximality Conjecture, \emph{Discrete Analysis}, 2 (2022) 50p.

\bibitem{CJT1} K. Clinch, B. Jackson and S. Tanigawa, Abstract 3-Rigidity and Bivariate $C_2^1$-Splines II: Whiteley’s Maximality Conjecture, \emph{Discrete Analysis}, 3 (2022) 35p.


\bibitem{C&H} C.R. Coullard and L. Hellerstein, Independence and port oracles for matroids, with an application to computational learning theory, \emph{Combinatorica}, 16 (1996) 189--208.

\bibitem{CW} R. Connelly, W. Whiteley, Global rigidity: The effect of coning, \emph{Discrete \& Computational Geometry} 43:4 (2010) 717--735.

\bibitem{DL} 
A. Dress and L. Lov\'asz, On some combinatorial properties of algebraic matroids, \emph{Combinatorica}, 7 (1987) 39--48.

\bibitem{EJNSTW} Y. Eftekhari, B. Jackson, A. Nixon, B. Schulze, S. Tanigawa and W. Whiteley,
Point-hyperplane frameworks, slider joints and rigidity preserving
transformations, {\em Journal of Combinatorial Theory: Series B}, 135 (2019) 44--74.

\bibitem{GGJ} D. Garamv{\"o}lgyi, S. Gortler and T. Jord\'{a}n, Globally rigid graphs are fully reconstructible, \emph{Forum of Mathematics, Sigma}, 10 (2022).

\bibitem{GJ} D. Garamv{\"o}lgyi and T. Jord\'{a}n, Minimally globally rigid graphs, \emph{European Journal of Combinatorics}, 108 (2023) 103626.

\bibitem{Gstress} D. Garamv{\"o}lgyi, Stress-linked pairs of vertices and the generic stress matroid, preprint, (2023) arXiv:2308.16851.

\bibitem{GGJN} G. Grasegger, H. Guler, B. Jackson and A. Nixon, Flexible circuits in the $d$-dimensional rigidity matroid, \emph{Journal of Graph Theory}, 100:2, (2021) 315--330.

\bibitem{GGJNc} G. Grasegger, H. Guler, B. Jackson and A. Nixon, Corrigendum to Flexible circuits in the $d$-dimensional rigidity matroid, \emph{Journal of Graph Theory}, 103:2, (2022) 307--308.

\bibitem{Gra} J. Graver, Rigidity matroids, \emph{SIAM Journal on Discrete Mathematics}, 4 (1991) 355--368.

\bibitem{GSS} J. Graver, B. Servatius and H. Servatius, Combinatorial rigidity, American Mathematical Society, Graduate Studies in Mathematics, Providence, RI, 1993.

\bibitem{JJ}
B. Jackson and T. Jord\'an, The $d$-dimensional rigidity matroid of sparse graphs,  \emph{Journal of Combinatorial Theory: Series B}, 95 (2005) 118--133.

\bibitem{JNS} B. Jackson, A. Nixon and B. Smith, The $k$-fold circuit property for matroids, preprint, (2024) arXiv:2412.14782.

\bibitem{laman} G. Laman, On graphs and rigidity of plane skeletal structures, \emph{Journal of Engineering Mathematics} 4 (1970)
331--340.

\bibitem{Lov} L. Lov\'{a}sz, Matroid matching and some applications, \emph{Journal of Combinatorial Theory: Series B}, 28 (1980) 208--236.

\bibitem{Mak} M. Makai, Matroid matching with Dilworth truncation, \emph{Discrete Mathematics}, 308 (2008) 1394--1404.

\bibitem{Max} J.~C.~Maxwell, On the calculation of the equilibrium and stiffness of frames, \emph{The London, Edinburgh, and Dublin Philosophical Magazine and Journal of Science}, 27:182 (1864) 294--299.

\bibitem{Oxl} J. Oxley, Matroid Theory, Oxford Science  Publications, The Clarendon Press, Oxford University Press, NewYork, 1992.

\bibitem{pol} H. Pollaczek-Geiringer, Uber die Gliederung ebener Fachwerke, \emph{ZAMM - Journal of Applied Mathematics and Mechanics / Zeitschrift fur Angewandte Mathematik und Mechanik}, 7 (1927), 58-72 and 12 (1932) 369-376.

\bibitem{Ser} B. Servatius and H. Servatius, On the 2-sum in rigidity matroids, \emph{European Journal of Combinatorics}, 32 (2011) 931--936.

\bibitem{Tay} T.-S. Tay, On generically dependent bar frameworks in space, \emph{Structural Topology}, 20 (1993) 27--48.

\bibitem{TW} T.-S. Tay and W. Whiteley, Generating isostatic frameworks, \emph{Structural Topology}, 11 (1985) 20--69.

\bibitem{TLV} P. Tong, E. L. Lawler and V. V. Vazirani, Solving the weighted parity problem for gammoids by reduction to graphic matching, \emph{Progress in combinatorial optimization, Conf. Waterloo/Ont. 1982}, (1984) 363--374.

\bibitem{Wcone} W. Whiteley, Cones, infinity and one-story buildings, \emph{Structural Topology}, 8 (1983) 53--70.

\bibitem{Wlong} W. Whiteley, Some matroids from discrete applied geometry,  in \emph{Matroid Theory}, J. E. Bonin, J. G. Oxley, and B.
Servatius eds., Contemporary
Mathematics 197, American Mathematical Society,  1996, 171--313.
\end{thebibliography}
\end{document}